\documentclass[reqno,11pt]{article}

\usepackage[utf8,latin1]{inputenc}
\usepackage{amsbsy}
\usepackage{amsfonts}
\usepackage{amsmath}
\usepackage{amsthm}
\usepackage{amssymb}
\usepackage{enumerate}
\usepackage{hyperref}
\usepackage{color}

\newcommand{\lcirc}{{\raise-0.15ex\hbox{$\scriptscriptstyle \circ$}}}
\newcommand{\lstar}{{\raise-0.15ex\hbox{$\scriptstyle \ast$}}}
\newcommand{\mathds}{\mathbf}
\newcommand{\tphh}{\hat\varphi^{ \odot}}

\newtheorem{theorem}{Theorem}
\newtheorem{lemma}[theorem]{Lemma}
\newtheorem{proposition}[theorem]{Proposition}

\newtheorem{remark}[theorem]{Remark}

\begin{document}
\title{Uniform point variance bounds in classical beta ensembles}
\author{Joseph Najnudel, B\'alint Vir\'ag}
\date{}
\maketitle

\begin{abstract}

In this paper, we give bounds on the variance of the number of points of the Circular and the Gaussian $\beta$ Ensemble in arcs of the unit circle or intervals of the real line. These bounds are logarithmic with respect to the renormalized length of these sets, which is expected to be optimal up to a multiplicative constant depending only on $\beta$. 
\end{abstract}
\section{Introduction}
 In the present article, the two following ensembles are considered:
\begin{itemize}
\item The Circular $\beta$ Ensemble, which consists in  a set of $n$ random points $\lambda_1, \dots, \lambda_n$ on the unit circle, whose joint distribution has a density of the form
$$Z_{n,\beta}^{-1} \prod_{1 \leq j < k \leq n} |\lambda_j - \lambda_k|^{\beta}$$
with respect to the Lebesgue measure on the $n$-th power of the unit circle. 
\item The Gaussian $\beta$ Ensemble, for which the points are on the real line, with density of the form 
$$(Z'_{n,\beta})^{-1} e^{- (\beta/4) \sum_{1 \leq j \leq n} \lambda_j^2} \prod_{1 \leq j < k \leq n} |\lambda_j - \lambda_k|^{\beta}$$
with respect to the Lebesgue measure on $\mathbb{R}^n$. 
\end{itemize}
These ensembles are defined for all positive values of $\beta$, and they correspond to the law of the spectrum of random matrices. 
For $\beta \in \{1,2,4\}$, we get the spectrum of the Circular Orthogonal ($\beta = 1$), Unitary ($\beta = 2$) and Symplectic ($\beta  = 4$) Ensembles, and the Gaussian Orthogonal, Unitary and Symplectic Ensembles, respectively. 
For general $\beta$, Gaussan and Laguerre ensembles have been constructed as spectra of random matrices by Dumitriu and Edelman \cite{DE02} (see also Trotter \cite{Tro84}), and random unitary matrices whose spectrum follows the Circular $\beta$ ensembles have 
been contructed by Killip and Nenciu \cite{bib:KN04}.  

In this article, we study the fluctuations of the distribution of the number of points lying in a given arc (in the Circular case) or a given interval (in the Gaussian case). 
These fluctuations have been first studied in the particular cases $\beta \in \{1,2,4\}$, where the correlation functions of the point processes are explicitly given by exact determinantal or Pfaffian formulas. In \cite{CL95}, Costin and Lebowitz proved that the number of eigenvalues of the Gaussian Orthogonal, Circular or Symplectic Ensemble in an interval
  has Gaussian asymptotic fluctuations when the average number of points tends to infinity with the dimension, the variance being logarithmic with respect to the mean. This result has been extended to more general determinantal point processes by Soshnikov in \cite{Sos00} and \cite{Sos02},  and to more general linear statistics of the eigenvalues: if the test function of the linear stastistics is sufficiently smooth, we can get central limit theorems without normalization, which is unusual in probability theory. This case occurs in particular when we 
consider smooth linear statistics of the Circular Unitary Ensemble, as in Diaconis and Shahshahani \cite{DS94}, and in Diaconis and Evans \cite{DE01}. Moreover, a  central limit theorem has been proven by Gustavsson \cite{Gus05}, for the joint distribution of the position of  finitely many individual eigenvalues of the Gaussian Unitary Ensemble. 

The case of general $\beta$ has been studied later than the case $\beta \in \{1,2,4\}$, and it is much more difficult because no convenient formulas are known for the correlation functions of the point processes.  For the Circular $\beta$ Ensemble, Killip \cite{Kil08} has proven a central limit theorem for the 
number of points in given arcs, the variance being logarithmic in the dimension, and another central limit theorem, with no normalization, has been obtained by Jiang and Matsumoto \cite{JM15} for smooth linear statistics.  For the Gaussian $\beta$ Ensemble and some of its generalizations, central limit theorems have been obtained for  smooth linear statistics, for example by Johansson \cite{Joh98}, by Shcherbina \cite{Shc13}, or by Bekerman, Lebl\'e and Serfaty \cite{BLS18}. Rigidity and a mesoscopic central limit theorem has also been obtained for the Dyson Brownian motion in a paper by Huang and Landon \cite{Huang2018}. 

However, it seems that similar results are not known for the number of points of the Gaussian $\beta$ Ensemble lying in a given interval. In the present paper, 
we prove a bound on the variance of the number of points in intervals, for both the Circular $\beta$ Ensemble and the Gaussian $\beta$ Ensemble. Our result does not provide a central limit theorem, in particular, it does not imply the result by Killip \cite{Kil08}. However, our bound is unconditionally available for all intervals and all values of $n$: it covers microscopic, mesoscopic and macroscopic scales and it does not need that we take a limit when $n$ goes to infinity. If we rescale the interval or the arc in such a way that the average spacing between the points has  order $1$, then the bound we get is logarithmic in the length of the interval we consider, which we expect to be optimal up to a multiplicative constant depending only on $\beta$. 
Moreover, we deduce similar bounds  for the scaling limit of the Circular and the Gaussian $\beta$ Ensemble, called the $\operatorname{Sine}_{\beta}$ point process. 
The existence of a scaling limit has been proven by Killip and Stoiciu in \cite{bib:KSt09} in the Circular case, and by Valk\'o and Vir\'ag \cite{VV} in the Gaussian case: Nakano \cite{Nakano} has then proven that the two scaling limits are the same. 

  The bounds proven in the present article play a crucial role in our companion paper \cite{NV19}, where we construct interlacing  $\operatorname{Sine}_{\beta}$ point processes generalizing  the bead process introduced by Boutillier \cite{Bou09}. Our bounds are also used by Huang \cite{Huang19} in his study of the eigenvalues of the minors of Wigner matrices. 

The proofs of the present article use the same tools as the papers by Killip and Stoiciu \cite{bib:KSt09}, and by Valk\'o and Vir\'ag \cite{VV}: the theory of the Orthogonal Polynomials on the Unit Circle in the Circular case,  the tridiagonal random matrix model by Dumitriu and Edelman and a discrete version of the Brownian carousel in the Gaussian case. 
Our estimates related to the Circular $\beta$ Ensemble and the $\operatorname{Sine}_{\beta}$ process are proven in Section 2. The estimates for the Gaussian $\beta$ Ensemble 
are proven in Section 3, up to four key estimates whose proof is very technical and postponed to Section 4. 
 The Gaussian case is indeed much more difficult to handle than the Circular case. 
\section{Estimates  for the Circular Beta Ensemble and the $\operatorname{Sine}_{\beta}$ process}
Here, we define the  $\operatorname{Sine}_{\beta}$ point  process as the limit in law of the set of arguments of the points of the Circular $\beta$ Ensemble, multiplied by $n$. 
This limit in law has been proven to exist by Killip and Stoiciu \cite{bib:KSt09}. In \cite{VV}, Valk\'o and Virag prove the existence of a similar limit for the 
 Gaussian $\beta$ Ensemble, and in  \cite{Nakano}, Nakano shows that this limit is also $\operatorname{Sine}_{\beta}$. 
The main result of this section is the following:
\begin{theorem} \label{estimatequadraticCbeta}
The number of points of the Circular $\beta$ Ensemble of order $n$ in an arc $I$ of the unit circle 
has a variance bounded by $C_{\beta} \log (2 + n|I|)$, $|I|$ being the length of the arc and $C_{\beta} > 0$ depending only on $\beta$. 
Moreover, the variance of the number of points of the $\operatorname{Sine}_{\beta}$ process in an interval $I$ is bounded by  $C_{\beta} \log (2 + |I|)$. 
\end{theorem}
\begin{proof}
In our proof of the theorem, we start with the following result, which has been proven in 
 \cite{bib:KSt09}, by using the theory of the Orthogonal Polynomials on the Unit Circle.
 \begin{lemma}
 Let 
$(\gamma^{(n)}_j)_{0 \leq j \leq n-2}$ be random variables  on the unit
disc, whose density with respect to
the uniform probability measure is $(\beta/2)(n-j-1) (1-|\gamma^{(n)}_j|^2)^{(\beta/2)(n-j-1) - 1}$, and let 
$\eta$ be a uniform variable on $[0, 2\pi)$, independent
 of $(\gamma^{(n)}_j)_{0 \leq j \leq n-2}$. 
 We define the so-called {\it Pr\"ufer phases} $(\psi^{(n)}_k(\theta))_{\theta \in \mathbb{R}, 0 \leq k \leq n-1}$ as follows: 
 $\psi^{(n)}_0(\theta) = \theta$ and
for $0 \leq k \leq n-2$,
$$\psi^{(n)}_{k+1} (\theta) = \psi^{(n)}_k(\theta) + \theta +
2 \Im \log \left( \frac{1- \gamma^{(n)}_k}{1 - \gamma^{(n)}_k
e^{i \psi^{(n)}_k(\theta)}}  \right),$$
where one takes the principal branch of the logarithm, which is not ambiguous since one easily checks that the quantity inside the logarithm cannot be in $\mathbb{R}_-$. 
Then,  the random set
$$\{\theta \in \mathbb{R}, \psi^{(n)}_{n-1} (\theta)
\equiv \eta \; (\operatorname{mod } 2 \pi) \}$$
has the same law as the set of all determinations of the arguments of the $n$ points of a Circular $\beta$ Ensemble. 
\end{lemma}
In order to prove the first part of the theorem, it is enough (using rotational invariance of the Circular $\beta$ Ensemble) to bound the variance of the number of points in the arc
between $1$ and $e^{i x/n}$ by $C_{\beta} \log (2 + x)$ for all $x \in [0,2 \pi n)$. 

Since $\psi^{(n)}_{n-1}(0) = 0$, the lemma implies that the number of points $z \in [0,x]$ such that
$e^{i z /n}$ is in a given Circular Beta Ensemble
with $n$ points has the same law as the sum of
$\psi^{(n)}_{n-1}(x/n)/2 \pi$ and a random variable in
$[-1,1]$ which depends on the value of $\eta$.
Hence, it is sufficient to show the estimate:
$$\mathbb{E} \left[ \left(\psi^{(n)}_{n-1}(x/n) - x \right)^2 \right] = O ( \log (2+x)),$$
the implicit constant depending only on $\beta$. 

In order to prove this bound, we define, for
$\theta, a \in \mathbb{R}$, $(\psi^{(n)}_{k}(\theta,a))_{
0 \leq k \leq n-1}$ as the sequence satisfying the
same recursion as $(\psi^{(n)}_{k}(\theta))_{
0 \leq k \leq n-1}$, and such that
$\psi^{(n)}_0(\theta) = \theta + a$.
Since for any $\gamma$ in the unit disc,
$$\psi \mapsto \psi + 2 \Im \log \left(
\frac{1 - \gamma}{1 - \gamma e^{i \psi}} \right)$$
is increasing, we deduce that
$\psi^{(n)}_{k}(\theta,a)$ is increasing with respect to
$\theta$ and $a$. Moreover, the
 average of the function
$$z \mapsto \log  \left( \frac{1 - z}{1 - z e^{i \psi}} \right)$$
on any circle with center $0$ and radius strictly
smaller than $1$ is equal to zero, since
the function is holomorphic on the unit disc.
Hence,
$(\psi^{(n)}_{k}(\theta,a) - k \theta)_{0 \leq k \leq n-1}$
is a martingale for any $\theta$ and $a$. Moreover,
from the distribution of the variables
$(\gamma^{(n)}_k)_{n \geq 1, 0 \leq k \leq n-2}$, depending only on $n-k$, we deduce that for $0
\leq k_1  \leq k_2 \leq n-1$, and conditionally
on $(\psi^{(n)}_k(\theta,a))_{0 \leq k \leq k_1}$, $\psi^{(n)}_{k_1} (\theta, a) = \theta + b$, the
conditional law of $\psi^{(n)}_{k_2} (\theta,a)$
is equal to the law of $\psi^{(n-k_1)}_{k_2 - k_1}
(\theta, b)$. We then prove the following lemma:
\begin{lemma}
For $0 \leq k \leq n-1$, $0 \leq \theta \leq 1/n$,
$a \in \mathbb{R}$ and $b \geq 0$, one has
$$\mathbb{P} [\psi_k^{(n)}(\theta, a) \geq a + b]
\leq 12 \, e^{-b/12}.$$
\end{lemma}
\begin{proof}
One knows that $\psi_{k}^{(n)}(\theta,0) \geq \psi_k^{(n)}
(0,0) = 0$, and then by Markov's inequality and the
fact that $(\psi_{j}^{(n)}(\theta,0) - j \theta)_{0
\leq j \leq n-1}$ is a martingale,
$$\mathbb{P} [\psi_k^{(n)}(\theta,0) \geq 2\pi]
\leq \frac{1}{2 \pi}
\mathbb{E} [\psi_k^{(n)}(\theta,0)]
\leq \frac{k+1}{2 \pi n} \leq \frac{1}{2\pi}. $$
Let $\ell$ be a positive integer. If
$\psi_k^{(n)}(\theta,0) \geq 6 \pi \ell$,
let $T$ be the first index $k_0 \in \{0, \dots, k\}$ such that
$\psi_{k_0}^{(n)} (\theta, 0) \geq 6 \pi \ell$.
It is easy to check that the increments of
$(\psi_{j}^{(n)}(\theta,0))_{0
\leq j \leq n-1}$ are bounded by $2 \pi + 1/n \leq 4 \pi$, which implies that
$\psi_{T}^{(n)} (\theta, 0) \leq  \pi(6 \ell + 4)$.
Moreover, conditionally on
$T = k_0$ and $\psi_{T}^{(n)} (\theta, 0) =
b + \theta \leq \pi(6 \ell + 4)$, the conditional
law of
$\psi_k^{(n)}(\theta,0)$ is equal to
the law of $\psi_{k-k_0}^{(n-k_0)}(\theta, b)$.
Hence,
\begin{align*}\mathbb{P} [\psi_{k}^{(n)} (\theta, 0) \geq 6 \pi (\ell+1) | T = k_0, \psi_{T}^{(n)} (\theta, 0) =
b + \theta ]
& = \mathbb{P} [\psi_{k-k_0}^{(n-k_0)}(\theta, b)
\geq 6 \pi (\ell+1)]
\\ & \leq \mathbb{P} [\psi_{k-k_0}^{(n-k_0)}(\theta, \pi(6 \ell + 4))
\geq 6 \pi (\ell+1)]
\\ & = \mathbb{P} [\psi_{k-k_0}^{(n-k_0)}(\theta,0)
\geq 2 \pi] \leq \frac{1}{2 \pi}.
\end{align*}
In the last equality, we have used the general fact that:
$$\psi_{k}^{(n)}(\theta,a + 2\pi)
= \psi_{k}^{(n)} (\theta, a) + 2\pi$$
for all $n \geq 1$, $0 \leq k \leq n-1$, $a \in
\mathbb{R}$.
We then deduce:
$$\mathbb{P} [\psi_{k}^{(n)} (\theta, 0) \geq 6 \pi (\ell+1)] \leq \frac{1}{ 2\pi}
\mathbb{P} [\psi_{k}^{(n)} (\theta, 0) \geq 6 \pi \ell]$$
and by induction,
$$\mathbb{P} [\psi_{k}^{(n)} (\theta, 0) \geq 6 \pi \ell] \leq (2\pi)^{-\ell}.$$
Now, let $2 \pi \ell_1$ be the smallest multiple of
$2 \pi$ which is larger than or equal to $a$, and
let $6 \pi \ell_2$ be the largest multiple of
$6 \pi$ which is smaller than or equal to
$b - 2\pi$: in particular, $6 \pi \ell_2 \geq b - 8
\pi$. One deduces the lemma, as follows:
\begin{align*}
\mathbb{P} [\psi_{k}^{(n)} (\theta, a)
\geq a+b]
& \leq \mathbb{P} [\psi_{k}^{(n)} (\theta, 2 \pi \ell_1) \geq 2 \pi (\ell_1 -1) + b]
\\ & \leq  \mathbb{P} [\psi_{k}^{(n)} (\theta, 2 \pi \ell_1) \geq 2 \pi \ell_1 + 6 \pi \ell_2]
\\ &  = \mathbb{P} [\psi_{k}^{(n)} (\theta,0)
\geq 6 \pi \ell_2] \leq (2 \pi)^{-\ell_2}
\\ & \leq (2 \pi)^{- (b-8\pi)/6\pi}
\leq (2 \pi)^{4/3} e^{- b \log(2 \pi)/6 \pi}.
\end{align*}
\end{proof}
A consequence of the exponential tail of
the distribution of $\psi_k^{(n)}(\theta,a)$ is
a uniform bound on its variance:
\begin{lemma}
For $0 \leq k \leq n-1$, $0 \leq \theta \leq 1/n$,
$a \in \mathbb{R}$, the expectation
$\mathbb{E} [(\psi_k^{(n)}(\theta,a) - a)^2]$
is bounded by a universal constant.
\end{lemma}
\begin{proof}
Since $\theta > 0$, we deduce that $\psi_k^{(n)}(\theta,a)$ is larger than any multiple of $2\pi$
which is smaller than $a$, and
then larger than $a - 2 \pi$. Hence
\begin{align*}
\mathbb{E} [(\psi_k^{(n)}(\theta,a) - a)^2]
& \leq (2\pi)^2 \mathbb{P} [(\psi_k^{(n)}(\theta,a)
\leq a] +
\int_{0}^{\infty} 2 x \mathbb{P}[(\psi_k^{(n)}(\theta,a)  \geq a+x] \, dx
\\& \leq 4 \pi^2 + \int_0^{\infty}
24 x e^{-x/12} dx \leq 4 \pi^2 + 24 \cdot 12^2 \leq 3500.
\end{align*}
\end{proof}
Now, let us go back to
the proof of the theorem, and let us define $k$ as the infimum of
$n-1$ and $n- \lfloor n/(1+x) \rfloor$: in particular,
$0 \leq k \leq n-1$. Conditionally on
$\psi_k^{(n)} (x/n) = a + (x/n)$, the law of
$\psi_{n-1}^{(n)} (x/n)$ corresponds to the
law of $\psi_{n-k-1}^{(n-k)} (x/n, a)$. If
$n \geq 1  + x$, we have $n-k \leq n/(1+x)$, and
then $x/n \leq 1/(n-k)$. Hence, by the previous lemma, $\mathbb{E} [(\psi_{n-k-1}^{(n-k)} (x/n, a)
-a)^2]$, and then
$$\mathbb{E} [(\psi_{n-1}^{(n)} (x/n)  - a)^2 |
\psi_k^{(n)} (x/n) = a + (x/n)],$$
 are uniformly bounded. Hence, we have a uniform bound
for
$$\mathbb{E} [(\psi_{n-1}^{(n)} (x/n) -
\psi_k^{(n)} (x/n) + (x/n))^2 ],$$
and then for
$$\mathbb{E} [(\psi_{n-1}^{(n)} (x/n) -
\psi_k^{(n)} (x/n) )^2 ],$$
since $x/n \leq 1$ by assumption. The uniform bound
of the last quantity remains obviously true if
$n \leq 1+x$, since $k = n-1$ in this case.
Therefore, it is now sufficient to show the
bound $$\mathbb{E} [(\psi_k^{(n)} (x/n) - x)^2]
= O(\log (2+x)),$$
or equivalently
$$\mathbb{E} [(\psi_k^{(n)} (x/n) - (k+1)x/n)^2]
= O(\log (2+x)),$$
since $(n-k-1)x/n$ is uniformly bounded. Since
$(\psi_j^{(n)} (x/n) - (j+1)x/n)_{0 \leq j \leq n-1}$
is a martingale starting at zero, it is sufficient
to bound the expectation of the sum of the squared
increments:
$$\sum_{j=0}^{k-1} \mathbb{E} \left[
\left(\Im \log \left(
\frac{1 - \gamma_j^{(n)}}{1 - \gamma^{(n)}_j e^{i \psi}} \right)
\right)^2 \right]  = O (\log(2+x)). $$
Now, for $\gamma$ in the unit disc,
$$|\log (1 - \gamma) | \leq \sum_{\ell \geq 1}
|\gamma|^{\ell}/\ell = - \log (1 - |\gamma|),$$
and then
$$\mathbb{E} \left[
\left(\Im \log \left(
\frac{1 - \gamma_j^{(n)}}{1 - \gamma^{(n)}_j e^{i \psi}} \right)
\right)^2 \right]  \leq 4 \, \mathbb{E} [\log^2
(1 - |\gamma^{(n)}_j|)  ].$$
Since $|\gamma^{(n)}_j|^2$ is a Beta variable of
parameters $1$ and $\beta(n-j-1)/2$, we have
\begin{align*}
\mathbb{E} [\log^2
(1 - |\gamma^{(n)}_j|)  ] & = \beta(n-j-1)/2
\int_{0}^1  (1-x)^{(\beta(n-j-1)/2)-1} \log^2 (1 - \sqrt{x}) \, dx
\\ & =  \beta(n-j-1)/2
\int_{0}^1  y^{\beta(n-j-1)/2} \log^2 (1 - \sqrt{1-y}) \, \frac{dy}{y}
\\& =  \beta(n-j-1)/2
\int_0^{\infty}
e^{-u \beta(n-j-1)/2} \log^2 (1 - \sqrt{1-e^{-u}}) \, du. \end{align*}
Now, it is straightforward to check that
$$ \log^2 (1 - \sqrt{1-e^{-u}}) = O(u + u^2)$$
for $u \in \mathbb{R}_+^{*}$.
Hence,
\begin{align*}
\mathbb{E} [\log^2
(1 - |\gamma^{(n)}_j|)  ] &  = O\left(
 \int_0^{\infty}
u \beta(n-j-1) e^{-u \beta(n-j-1)/2} du  \right.
\\ & \left. +
 \int_0^{\infty}
u^2 \beta(n-j-1) e^{-u \beta(n-j-1)/2} du  \right)
\\ & = O \left( \frac{1}{\beta(n-j-1)} +
\frac{1}{\beta^2(n-j-1)^2} \right)
\\ & = O \left( \frac{1 + \beta}{\beta^2(n-j-1)}
\right).
\end{align*}
Adding this estimate for $j$ between $0$ and
$k-1$ gives a quadratric variation dominated,
with an implicit constant depending only on
$\beta$, by
$$\sum_{j = 0}^{k-1} \frac{1}{n-j-1}
= \sum_{j = n-k}^{n-1} \frac{1}{j}
 = \log[n/(n-k)] + O(1).$$
 If $n \geq 1 + x$, then
 $$n-k = \lfloor n/(1+x) \rfloor
 \geq \frac{n}{2(1+x)},$$
 hence
 $$\log [n/(n-k)] \leq \log (2+2x).$$
 If $n \leq 1+x$, then $n-k=1$ and
 $$\log [n/(n-k)] = \log n \leq \log (1+x). $$
 These estimates imply the first part of the theorem.

 Let us now show the second part, relative to the $\operatorname{Sine}_{\beta}$ point process. 
 We know that this point process is the scaling limit of the Circular $\beta$ Ensemble. 
 Hence,  by Skorokhod's representation theorem, one can construct point processes $L_n$, $L$, such that almost surely, the point measure corresponding
 to $L_n$ converges locally weakly to the measure corresponding to $L$, the distribution of $L$ and $L_n$ being given as follows: 
 \begin{itemize}
 \item The point process $L_n$ is obtained by taking all the determinations of the arguments of the $n$ points of a Circular $\beta$ Ensemble, multiplied by $n$. 
 \item The point process $L$ follows the $\operatorname{Sine}_{\beta}$ distribution. 
 \end{itemize}
 Let $x > 0$. Since $L$ almost surely does not contains the points $0$ and $x$, we have almost surely
 $$\operatorname{Card} (L_n \cap [0,x]) \longrightarrow  \operatorname{Card} (L \cap [0,x]),$$
 and then
 $$(\operatorname{Card} (L_n \cap [0,x]) - x/2 \pi)^2 \longrightarrow (\operatorname{Card} (L \cap [0,x]) - x/2 \pi)^2.$$
 By Fatou's lemma,
  $$\mathbb{E} [(\operatorname{Card} (L \cap [0,x]) - x/2 \pi)^2] \leq \underset{n \rightarrow \infty}{\lim \inf} \,
   \mathbb{E} [(\operatorname{Card} (L_n \cap [0,x])) - x/2 \pi)^2] \leq C_{\beta} \log(2 + x). $$
\end{proof}
A consequence of the previous result is the fact that
the points of a Circular Beta Ensemble are much more
regularly spaced than those of a Poisson point process. More precisely, we have the following result, used in our companion paper \cite{NV19}:
\begin{proposition}\label{estimatealmostsureCbeta}
With the notation above, 
for all $\alpha > 1/3$, there exists
 a random
variable $C > 0$,
stochastically dominated by a finite random variable depending only on
$\alpha$ and $\beta$, such that almost surely, for all $x \geq 0$,
$$|\operatorname{Card} (L_n \cap [0,x])
- x/2 \pi| \leq C (1+x)^{\alpha},$$
and
$$|\operatorname{Card} (L_n\cap [-x,0])
- x/2 \pi| \leq C (1+x)^{\alpha},$$
and the similar bounds with $L_n$ replaced by $L$. 
Moreover, we can take $C$ in such a way that the following holds for all $u > 0$: 
$$\mathbb{P} [ C \geq u ] \leq K_{\alpha, \beta} u^{-2} $$
for some $K_{\alpha, \beta} > 0$ depending only on $\alpha$ and $\beta$. 
\end{proposition}
\begin{remark}
The periodicity of $L_n$ implies that
$|\operatorname{Card} (L_n \cap [0,x])
- x/2 \pi|$ is  almost surely bounded when $x$ varies.
Hence, the result above becomes trivial if one allows
$C$ to depend on $n$.
Moreover, we expect that it remains true for any $\alpha > 0$, and not only for $\alpha > 1/3$.
\end{remark}
\begin{proof}
It is sufficient to check the first estimate. 
We prove the result for $L$: the proof of $L_n$ is exactly the same since we have the same estimates for the variance of the number of points in an interval. 
For any $p \geq 1$ and $A > 0$, one gets from the previous
theorem: 
\begin{align*}
\mathbb{P} [|\operatorname{Card} (L \cap [0,p^{3/2}]
- p^{3/2}/ 2\pi | \geq A p^{3\alpha/2} ]
&\leq A^{-2} p^{-3\alpha} \mathbb{E}
[(\operatorname{Card} (L \cap [0,p^{3/2}]
- p^{3/2}/ 2\pi )^2]
\\ &  = O(A^{-2}p^{-3\alpha} \log (1+p)),
\end{align*}
with an implicit constant depending only on $\alpha$  and $\beta$. By summing in $p$, one deduces that with probability
$1 - O(A^{-2})$,
$$|\operatorname{Card} (L \cap [0,p^{3/2}]
- p^{3/2}/ 2\pi | \leq A p^{3\alpha/2}$$
for all $p \geq 1$. In other words, the tail of
the infimum $B$ of the values $A$ satisfying the
previous bound is smaller than $B^{-2}$ times a
quantity  depending only on $\alpha$ and $\beta$. 
For any $x \geq 0$, let $p$ be the integer part
of $1 + x^{2/3}$: one has $(p-1)^{3/2} \leq x \leq
p^{3/2}$. Hence,
$$\operatorname{Card} (L \cap [0,x])  \leq
\operatorname{Card} (L \cap [0,p^{3/2}])
\leq p^{3/2}/2 \pi + Bp^{3\alpha/2}.$$
Now, it is immediate to check that
$x = p^{3/2} + O((1+x)^{1/3})$ (with a universal implicit constant), and that
$p^{3/2} = O(1+x)$. Hence
$$\operatorname{Card} (L \cap [0,x])
\leq x/2\pi + O((1+x)^{1/3}) + O(B(1+x)^{\alpha})
= x/2\pi +  O((1+B)(1+x)^{\alpha}),$$
with implicit constants depending only on $\alpha$.
Similarly,
\begin{align*}
\operatorname{Card} (L \cap [0,x])  \geq
\operatorname{Card} (L \cap [0,(p-1)^{3/2}])
& = (p-1)^{3/2}/2 \pi + O(Bp^{3\alpha/2})
\\& =  x/2\pi + O((1+x)^{1/3}) + O(B(1+x)^{\alpha})
\\& = x/2\pi +  O((1+B)(1+x)^{\alpha}).
\end{align*}
Hence, we are done, by taking $C$ equal to $1+B$ times a quantity depending only on $\alpha$: $C$ is
stochastically dominated by a variable depending
only on $\alpha$ and $\beta$, and the estimate on its  tail is deduced from the estimate on the tail of $B$ obtained above. 
\end{proof}

\section{Estimates for the Gaussian Beta Ensembles}

It is known that after suitable scaling, the empirical distribution of the points of the Gaussian Beta Ensemble tends to the semi-circle distribution. 
The  result below  gives a $L^2$ bound on the fluctuations of the number of points in an interval, with respect to this limiting distribution: 
\begin{theorem} \label{boundvariance2017}
For $-\infty \leq \Lambda_1 < \Lambda_2 \leq \infty$, let $N_n(\Lambda_1, \Lambda_2)$
be the number of points, between $\Lambda_1$ and $\Lambda_2$, of a Gaussian $\beta$ Ensemble with $n$
points, and let $N_{sc}(\Lambda_1, \Lambda_2)$
be $n$ times the measure of $(\Lambda_1, \Lambda_2)$
under the semi-circle distribution on the interval
$[- 2\sqrt{n}, 2 \sqrt{n}]$:
$$N_{sc}(\Lambda_1, \Lambda_2) :=
\frac{n}{2\pi} \int_{\Lambda_1/\sqrt{n}}^{\Lambda_2/\sqrt{n}}
\sqrt{(4 - x^2)_+}  \,  dx.
$$
Then,
$$\mathbb{E} [(N_n(\Lambda_1, \Lambda_2) - N_{sc}(\Lambda_1, \Lambda_2))^2] = O( \log (2 + (\sqrt{n}(\Lambda_2 - \Lambda_1) \wedge n))),$$
 where, here and in the sequel of the article, we use the notation: 
$$a \wedge b := \min(a,b), \quad a \vee b  := \max(a,b).$$
\end{theorem}

The theorem is much more difficult to prove than the previous estimates on the Circular $\beta$ Ensemble. 

In Trotter \cite{Tro84} and Dumitriu and Edelman \cite{DE02} are
introduced some ensembles of tridiagonal real symmetric matrices,
for which  the distribution of the eigenvalues
corresponds to the Gaussian $\beta$ Ensemble.

The tridiagonal real symmetric random matrices $(M_{p,q})_{1 \leq p,
q \leq n}$ can be described as follows:
the diagonal entries  $(M_{p,p})_{1 \leq p \leq n}$
are centered Gaussian variables of variance $2/\beta$, the
entries just above the diagonal $(M_{p,p+1})_{1
\leq p \leq n-1}$ are $\chi_{\beta(n-p)}/ \sqrt{\beta}$, $\chi_{m}$
being a chi-distributed random variable with $m$ degrees of freedom, all these entries being independent.

In \cite{VV}, the authors prove that
after a suitable rescaling, the limiting distribution
of the eigenvalues of this matrix ensemble tends
to the $\mathrm{Sine}_{\beta}$ point process. 

The general method used in the article consists
 of the following: let $\Lambda$ be an eigenvalue of
 a tridiagonal matrix whose distribution is
 given above,
 and $(u_{\ell})_{1 \leq \ell \leq n}$ an
 eigenvector corresponding to this eigenvalue.
 Solving the eigenvalue equation gives a three term
 recursion for the sequence $(u_{\ell})_{1 \leq \ell \leq n}$, which in turn implies that the ratio
 $r_{\ell} = u_{\ell + 1}/ u_{\ell}$ (considered
 as an element of the projective real line) satisfies
 a recursion of the form $r_{\ell + 1}
  = r_{\ell} .
  \mathbf{R}_{\ell,
   \Lambda}$, where
  $r_{\ell}.\mathbf{R}_{\ell, \Lambda}$ denotes
  the image of $r_{\ell}$ by a transformation
   $\mathbf{R}_{\ell, \Lambda}$ of the projective real line, given by $r.\mathbf{R}_{\ell, \Lambda} =   b - a/r$, $a, b \in \mathbb{R}$ depending on $\Lambda$ and on the entries of the matrix. This recursion can be followed for any
   $\Lambda \in \mathbb{R}$: however, the boundary
   conditions are consistent only if $\Lambda$ is
   an eigenvalue of the matrix. Indeed, although
   $r_{\ell}$ is originally defined only for
   $\ell \in \{1, \dots, n-1\}$, one can extend
   this notation by considering that the entries
   $u_0$ and $u_n$ are equal to zero: this gives
   $r_0 = \infty$ and $r_n = 0$. On the other hand,
   one can naturally define $\mathbf{R}_{\ell, \Lambda}$ also for $\ell = 0$ and $\ell = n-1$ and then follow the recursion from $r_0$ to $r_n$:
   $\Lambda$ is then an eigenvalue if and only if
   this recursion is consistent with the boundary
   conditions $r_0 = \infty, r_n = 0$, i.e.:
   \begin{equation} \label{eigeq}
\infty. \mathbf{R}_{0, \Lambda}.
   \mathbf{R}_{1, \Lambda} \dots \mathbf{R}_{n-1, \Lambda} = 0,
\end{equation} 
this notation meaning that the image of $\infty$ by successive applications of the transformations $ \mathbf{R}_{0, \Lambda},  \dots,  \mathbf{R}_{n-1, \Lambda}$, applied in this order, is equal to $0$. 
   In \cite{VV}, Subsection 4.2, the tridiagonal model
   described above is slightly modified by a conjugation with a suitably chosen diagonal matrix,
   which does not change the eigenvalues.
  The new model consists of a (non-Hermitian) tridiagonal
  matrix $(\widetilde{M}_{p,q})_{1 \leq p, q \leq n}$,
  for which:
  \begin{itemize}
  \item $\widetilde M_{p,p} = X_{p-1}$ for all $p \in \{1, \dots,
  n\}$.
  \item $\widetilde M_{p+1, p} = s_p$ for all    $p \in \{1, \dots, n-1\}$.
  \item $\widetilde M_{p, p+1} = s_{p-1} + Y_{p-1}$ for all  $p \in \{1, \dots, n-1\}$.
  \end{itemize}
  Here, $s_p = \sqrt{n - p - 1/2}$ for all
  $p \in \{0, \dots, n-1\}$, and $(X_p)_{0 \leq p \leq n-1}$, $(Y_p)_{0 \leq p \leq n-2}$ are independent random variables whose law is given by 
$$ X_p = \mathcal{N} (0, 2 / \beta),  \quad Y_p = \frac{\chi^2_{(n-p-1)\beta}}{\beta s_{p+1}} - s_p,$$
which implies that 
  $$\mathbb{E} [X_p] = O((n-p)^{-3/2}), \;
  \mathbb{E} [X^2_p] =  \frac{2}{\beta} +
  O((n-p)^{-1}), \;  \mathbb{E} [X^4_p] = O(1),$$
  and the same estimates  for the moments of 
  $(Y_p)_{0 \leq p \leq n-2}$.

 The interest of this change of matrix model is
 the independence of the different rows, which implies the independence of the random maps $(\mathbf{R}_{\ell,
 \Lambda})_{0 \leq \ell \leq n-1}$.

Writing in detail the eigenvalue equation corresponding to the row $\ell +1$ of the matrix $\widetilde{M}$, we get the following formula (see \cite{VV}, equation (44)): 
 $$\mathbf{R}_{\ell,
 \Lambda} = \mathbf{Q} (\pi) \mathbf{A} (1,
 \Lambda/s_{\ell}) \mathbf{W}_{\ell},$$
 where for $\theta \in \mathbb{R}$,
 $x \in \mathbb{P}^1 (\mathbb{R})$,
 $$x . \mathbf{Q} (\theta) = \frac{x \cos(\theta/2)
 + \sin(\theta/2) }{-x \sin(\theta/2)
 + \cos(\theta/2)},$$
 in particular, $x . \mathbf{Q} (\pi)
 = -1/x$, and where for $a \in \mathbb{R}_+^*$, $b \in \mathbb{R}$,
 $$x. \mathbf{A} (a,b) = a(x+b)$$
 and $$ \mathbf{W}_{\ell} =
  \mathbf{A}((1+ Y_{\ell}/s_{\ell})^{-1}, - X_{\ell}
  / s_{\ell}).$$
  Note that, similarly as in \eqref{eigeq}, the composition is performed from the left
  to the right (one applies $\mathbf{Q} (\pi)$,
  then $\mathbf{A} (1,
 \Lambda/s_{\ell})$ and then $\mathbf{W}_{\ell}$), and that all the randomness is contained in the
  factor $\mathbf{W}_{\ell}$.

  On the other hand, the projective line $\mathbb{P}^1(\mathbb{R})$ can be identified with the unit circle
  $\mathbb{U}$, via the
  {\it Cayley transform} $\mathbf{U}$, given
  by:
  $$x. \mathbf{U}  =  \frac{i-x}{i+x}.$$
  Hence, via a conjugation by the Cayley transform, for $\theta \in \mathbb{R}$, $a \in
  \mathbb{R}_+^*$ and $b \in \mathbb{R}$,
  $\mathbf{Q} (\theta)$ and $\mathbf{A} (a,b)$
   can be identified
  with bijections of the unit circle instead of the
  projective line: moreover, one checks that
  $\mathbf{Q} (\theta)$ corresponds to a
  rotation of angle $\theta$ (see \cite{VV}, Subsection 4.1 for more detail). 
  Similarly, $\mathbf{R}_{\ell, \Lambda}$ can be seen as a transformation of 
unit circle, and the image of $z \in
  \mathbb{U}$ will be denoted:
  $$z \lcirc \mathbf{R}_{\ell, \Lambda}
  = z \lcirc \mathbf{Q} (\pi) \mathbf{A} (1,
 \Lambda/s_{\ell}) \mathbf{W}_{\ell}.$$
More generally and more rigorously, for a transformation $\mathbf{R}$ of the projective real line, and for $z \in \mathbb{U}$, we denote by
$z \lcirc \mathbf{R} $  the image of $z$ by the composition of the inverse Cayley transform $\mathbf{U}^{-1} : \mathbb{U} \rightarrow \mathbb{P}^1(\mathbb{R})$, 
the map $\mathbf{R}  : \mathbb{P}^1(\mathbb{R}) \rightarrow  \mathbb{P}^1(\mathbb{R})$, and the Cayley transform $\mathbf{U} : \mathbb{P}^1(\mathbb{R}) \rightarrow \mathbb{U} $. 

  In this setting, $\Lambda$ is an eigenvalue of $M$ or $\widetilde{M}$ if and only if
    $$ (-1) \lcirc \mathbf{R}_{0, \Lambda}
   \mathbf{R}_{1, \Lambda} \dots \mathbf{R}_{n-1, \Lambda} = 1.$$

   Now, the unit circle can be lifted to the real line,
 by taking the argument: at each point
 $z \in \mathbb{U}$, one associates all the values $x\in \mathbb{R}$ such that $e^{i x} = z$, which gives
 a $2 \pi$-periodic subset of $\mathbb{R}$.
 The applications $\mathbf{Q} (\theta)$ and
 $\mathbf{A} (a,b)$ can then be defined as actions on the
 real line $\mathbb{R}$, as follows: for all $\varphi
 \in \mathbb{R}$, the image of $\varphi$ by $\mathbf{Q} (\theta)$
 is $\varphi \lstar \mathbf{Q} (\theta) = \varphi + \theta$,
 and $\mathbf{A} (a,b)$ is the unique continuous and increasing
 application on $\mathbb{R}$ such that
 $$e^{i \varphi \lstar \mathbf{A} (a,b)}
 = e^{i \varphi} \lcirc \mathbf{A} (a,b)$$
 for all $\varphi \in \mathbb{R}$, and $\pi \lstar
 \mathbf{A} (a,b) = \pi$: note that this last equality
 is possible since $\infty.  \mathbf{A} (a,b) = \infty$, which
 implies that $(-1) \lcirc \mathbf{A} (a,b) = -1$, i.e.
 $e^{i \pi} \lcirc \mathbf{A} (a,b) = e^{i \pi}$.
More generally, for all $x \in \mathbb{R}$, and for any  $\mathbf{T}$  explicitly written as 
$$\mathbf{T} = \mathbf{T}_1 \mathbf{T}_2 \dots \mathbf{T}_m,$$
 where for all $j \in \{1,\dots, m\}$,  $\mathbf{T}_j$ is equal to $\mathbf{A}(a,b)$ (for some $a \in \mathbb{R}_+^*$, $b \in \mathbb{R}$) or $\mathbf{Q}(\theta)$ (for some $\theta \in \mathbb{R}$), we denote by $x \lstar \mathbf{T}$ the image of $x$ by the composition of the maps $x \mapsto x \lstar \mathbf{T}_1$, ..., $x \mapsto x \lstar \mathbf{T}_m$, in this order. 
The examples of $\mathbf{Q}(\theta)$ and $\mathbf{Q}(\theta + 2 \pi)$ show that  the same map from $\mathbb{P}^1 (\mathbb{R})$  to $\mathbb{P}^1 (\mathbb{R})$ may correspond to different maps from $\mathbb{R}$ and $\mathbb{R}$ after applying 
the Cayley transform and taking the argument: however, all these maps differ by a translation of a multiple of $2 \pi$. 
Since $\mathbf{W}_{\ell}$ is a (random) map of the form $\mathbf{A} (a,b)$, we can define: 
 $$x \lstar \mathbf{R}_{\ell, \Lambda}
 = x \lstar \mathbf{Q} (\pi) \mathbf{A} (1,
 \Lambda/s_{\ell}) \mathbf{W}_{\ell}$$
 for all $x  \in \mathbb{R}$. The following equality holds: 
\begin{equation}(2\pi + x) \lstar \mathbf{R}_{\ell, \Lambda}
 = 2\pi + (x \lstar \mathbf{R}_{\ell, \Lambda} ).
 \label{quasiperiodicity}
 \end{equation}
 and  $\Lambda$ is an eigenvalue of $\widetilde{M}$ if and only
 if
 $$ \pi \lstar \mathbf{R}_{0, \Lambda} \mathbf{R}_{1, \Lambda}
 \dots   \mathbf{R}_{n-1, \Lambda} \in 2 \pi \mathbb{Z}.$$
 Using \eqref{quasiperiodicity}, we deduce that
 for any $\ell \in \{0, \dots, n\}$, this condition is
 equivalent to
 $$\hat{\varphi}_{\ell, \Lambda} - \tphh_{\ell, \Lambda}
 \in 2 \pi  \mathbb{Z}$$
 where
 $$ \hat{\varphi}_{\ell, \Lambda} =   \pi \lstar \mathbf{R}_{0, \Lambda} \mathbf{R}_{1, \Lambda}
 \dots   \mathbf{R}_{\ell-1, \Lambda},$$
 and 
 $$\tphh_{\ell, \Lambda} = 0 \lstar \mathbf{R}_{n-1, \Lambda}^{-1} \mathbf{R}_{n-2, \Lambda}^{-1}
 \dots   \mathbf{R}_{\ell, \Lambda}^{-1}$$
is the image of $0$ by  successive application
of the inverses of the maps $x \mapsto x \lstar \mathbf{R}_{n-1, \Lambda}$, $x \mapsto x \lstar \mathbf{R}_{n-2, \Lambda}$, ..., $x \mapsto x \lstar \mathbf{R}_{\ell, \Lambda}$.
 Since $x \mapsto x \lstar \mathbf{R}_{\ell, \Lambda}$ is an analytic and
 increasing function, and is also strictly increasing with respect to $\Lambda$,
 one deduces that for $1 \leq \ell \leq n-1$, $\hat{\varphi}_{\ell, \Lambda}$ is
 analytic and strictly increasing with respect to
 $\Lambda$, and that
 $\tphh_{\ell, \Lambda}$ is analytic and strictly decreasing in $\Lambda$.

 Moreover, one checks that for $\Lambda$ going to $-\infty$,
$\hat{\varphi}_{\ell, \Lambda}$ tends to $\pi$ and
$\tphh_{\ell, \Lambda}$ tends to $0$. Hence, one deduces the
following result:
\begin{proposition} \label{numbereigenvalues}
The number of eigenvalues of $M$ in the interval
$(-\infty, \Lambda]$ is equal to the integer part
of $\frac{1}{2\pi} (\hat{\varphi}_{\ell, \Lambda}
- \tphh_{\ell, \Lambda})$.
\end{proposition}
For any $\mu \geq 0$, following \cite{VV},
we introduce the following quantities:
$$ n_0 (\mu) = \max \left(n - \frac{\mu^2}{4} - \frac{1}{2}, 1 \right),$$
 and for all $\ell$ such that $0 \leq \ell < n_0(\mu)$,
 $$\rho_{\ell} (\mu) :=
 \sqrt{\frac{\mu^2/4}{\mu^2/4 + n_0(\mu) - \ell}}
 + i \, \sqrt{\frac{n_0(\mu) - \ell}{\mu^2/4 + n_0(\mu) - \ell}}.$$
 When there is no ambiguity, we will write respectively
 $n_0$ and $\rho_{\ell}$: notice that $|\rho_{\ell}| = 1$.
 Then, we introduce a modification
  of the phase $\hat{\varphi}_{\ell, \Lambda}$, denoted
  $\varphi_{\ell, \Lambda, \mu}$, in order to remove the
  fast variations of $\hat{\varphi}_{\ell, \Lambda}$, which
 come from the deterministic  part
   $\mathbf{Q} (\pi) \mathbf{A} (1,
 \Lambda/s_{\ell})$ of $\mathbf{R}_{\ell, \Lambda}$. The precise definition of $\varphi_{\ell, \Lambda, \mu}$ is the
 following:
 $$\varphi_{\ell, \Lambda, \mu} =
  \hat{\varphi}_{\ell, \Lambda} \lstar
  \mathbf{A} (\Im(\rho_{\ell})^{-1}, - \Re (\rho_{\ell}))
   - 2 \sum_{j = 0}^{\ell - 1} ( \pi - \operatorname{Arg} (\rho_j)).$$
The following
   result corresponds to Proposition 18 in \cite{VV}. For sake of completeness,  we prove this result here. Our proof also explains the introduction of the term $\pi$ in the sum just above, which does not appear explicitly in \cite{VV}, and which is related to the choice of the determination of the argument of some complex numbers of modulus $1$. 

  \begin{proposition} \label{ash} 
  For $\Lambda \in \mathbb{R}$ and $\mu \geq 0$, one has $\varphi_{0, \Lambda, \mu} = \pi$ and
   for $0 \leq \ell < n_0 -1$:
  $$\Delta \varphi_{\ell, \Lambda, \mu}
  := \varphi_{\ell+1 , \Lambda, \mu} -\varphi_{\ell, \Lambda, \mu} = \operatorname{ash} (\mathbf{S}_{\ell, \Lambda, \mu},
  -1, e^{i \varphi_{\ell, \Lambda, \mu}} \overline{ \eta_{\ell}}),$$
  where
$$  \mathbf{S}_{\ell, \Lambda, \mu}
= \mathbf{A}^{-1} (\Im(\rho_{\ell})^{-1}, - \Re (\rho_{\ell}))
\mathbf{A} \left(1, \frac{\Lambda - \mu}{s_{\ell}}\right) \mathbf{W}_{\ell} \mathbf{A} (\Im(\rho_{\ell+1})^{-1}, - \Re (\rho_{\ell+1})),$$
  $$ \operatorname{ash} (\mathbf{S}_{\ell, \Lambda, \mu},
  e^{ix}, e^{iy}) := (y \lstar \mathbf{S}_{\ell, \Lambda, \mu} -x \lstar \mathbf{S}_{\ell, \Lambda, \mu})
  - (y-x)$$
  and
  $$\eta_{\ell} = \prod_{j=0}^{\ell} \rho_j^2.$$
 Note that the definition of the angular shift $\operatorname{ash} $ is meaningful since the right-hand side
  does not depend on the determination of the arguments of $e^{ix}$ and $e^{iy}$.
  \end{proposition}
\begin{proof}
We have $$\varphi_{0, \Lambda, \mu} =  \hat{\varphi}_{0, \Lambda} \lstar
  \mathbf{A} (\Im(\rho_{0})^{-1}, - \Re (\rho_{0})) = \pi \lstar  \mathbf{A} (\Im(\rho_{0})^{-1}, - \Re (\rho_{0})) = \pi.$$
Moreover, for  $0 \leq \ell < n_0 -1$, 
\begin{align*}
& \varphi_{\ell+1 , \Lambda, \mu}   =   \hat{\varphi}_{\ell+1, \Lambda} \lstar
  \mathbf{A} (\Im(\rho_{\ell+1})^{-1}, - \Re (\rho_{\ell+1}))
   - 2 \sum_{j = 0}^{\ell} ( \pi - \operatorname{Arg} (\rho_j))
\\ & =  \hat{\varphi}_{\ell, \Lambda} \lstar \mathbf{R}_{\ell, \Lambda}  \mathbf{A} (\Im(\rho_{\ell+1})^{-1}, - \Re (\rho_{\ell+1}))
   - 2 \sum_{j = 0}^{\ell} ( \pi - \operatorname{Arg} (\rho_j))
\\ & =  \hat{\varphi}_{\ell, \Lambda} \lstar \mathbf{Q}(\pi) \mathbf{A} (1, \mu/s_{\ell})  \mathbf{A} (1, (\Lambda - \mu)/s_{\ell})  \mathbf{W}_{\ell}  \mathbf{A} (\Im(\rho_{\ell+1})^{-1}, - \Re (\rho_{\ell+1}))
   - 2 \sum_{j = 0}^{\ell} ( \pi - \operatorname{Arg} (\rho_j))
\\ & = \left(  \varphi_{\ell, \Lambda, \mu} +  2 \sum_{j = 0}^{\ell-1} ( \pi - \operatorname{Arg} (\rho_j)) \right) \lstar   \mathbf{A}^{-1} (\Im(\rho_{\ell})^{-1}, - \Re (\rho_{\ell}))
\mathbf{Q}(\pi) \mathbf{A} (1, \mu/s_{\ell}) \dots
\\ & \dots  \mathbf{A} (1, (\Lambda - \mu)/s_{\ell})  \mathbf{W}_{\ell}  \mathbf{A} (\Im(\rho_{\ell+1})^{-1}, - \Re (\rho_{\ell+1}))   - 2 \sum_{j = 0}^{\ell} ( \pi - \operatorname{Arg} (\rho_j))
\end{align*} 
Hence,
\begin{align*}
& \Delta \varphi_{\ell, \Lambda, \mu}  =  \left(  \varphi_{\ell, \Lambda, \mu} +  2 \sum_{j = 0}^{\ell-1} ( \pi - \operatorname{Arg} (\rho_j)) \right) \lstar   \mathbf{A}^{-1} (\Im(\rho_{\ell})^{-1}, - \Re (\rho_{\ell}))
\mathbf{Q}(\pi) \mathbf{A} (1, \mu/s_{\ell}) \dots
\\ & \dots  \mathbf{A} (1, (\Lambda - \mu)/s_{\ell})  \mathbf{W}_{\ell}  \mathbf{A} (\Im(\rho_{\ell+1})^{-1}, - \Re (\rho_{\ell+1}))   - \left( \varphi_{\ell, \Lambda, \mu} +  2 \sum_{j = 0}^{\ell} ( \pi - \operatorname{Arg} (\rho_j)) \right)
\\ &  = \left(  \varphi_{\ell, \Lambda, \mu} +  2 \sum_{j = 0}^{\ell} ( \pi - \operatorname{Arg} (\rho_j)) \right) \lstar  \mathbf{Q} ( -2 ( \pi - \operatorname{Arg} (\rho_{\ell}))) \mathbf{A}^{-1} (\Im(\rho_{\ell})^{-1}, - \Re (\rho_{\ell}))
\mathbf{Q}(\pi) \mathbf{A} (1, \mu/s_{\ell}) \dots
\\ & \dots  \mathbf{A} (1, (\Lambda - \mu)/s_{\ell})  \mathbf{W}_{\ell}  \mathbf{A} (\Im(\rho_{\ell+1})^{-1}, - \Re (\rho_{\ell+1}))   - \left( \varphi_{\ell, \Lambda, \mu} +  2 \sum_{j = 0}^{\ell} ( \pi - \operatorname{Arg} (\rho_j)) \right), 
\end{align*}
and then 
\begin{equation}
   \Delta \varphi_{\ell, \Lambda, \mu}   =  y \lstar \mathbf{V}  \mathbf{S}_{\ell, \Lambda, \mu} - y, \label{deltaxx}
\end{equation}
where $$e^{iy} =   e^{i \varphi_{\ell, \Lambda, \mu}} \overline{ \eta_{\ell}}$$
and 
$$\mathbf{V} =  \mathbf{Q} ( -2 ( \pi - \operatorname{Arg} (\rho_{\ell}))) \mathbf{A}^{-1} (\Im(\rho_{\ell})^{-1}, - \Re (\rho_{\ell}))
\mathbf{Q}(\pi)\mathbf{A} (1, \mu/s_{\ell})  \mathbf{A} (\Im(\rho_{\ell})^{-1}, - \Re (\rho_{\ell})).$$
Now, if we view $\mathbf{Q}(\pi)\mathbf{A} (1, \mu/s_{\ell}) $  as a transformation of the projective line, and if we extend it meromorphically to the Riemann sphere, then the image of
$\rho_{\ell}$ by this extended map is 
$$ - \frac{1}{\rho_{\ell}} +  \mu/s_{\ell} = - \frac{1}{\rho_{\ell}} + \frac{\mu}{ \sqrt{n - \ell -1/2}}  = - \frac{1}{\rho_{\ell}} + \frac{\mu}{\sqrt{\mu^2/4 + n_0 - \ell}}.$$
The last equality is due to the fact that 
$$n_0 = \left( n - \frac{\mu^2}{4}  - \frac{1}{2} \right) \vee 1$$
cannot be equal to $1$ since we assume $0 \leq \ell < n_0 -1$ in the proposition, and then $n_0 =n - \mu^2/4 - 1/2$. 
Since, by construction, $\rho_{\ell}$ is the solution of  the second degree equation 
$$x^2 -  \frac{\mu}{\sqrt{\mu^2/4 + n_0 - \ell}} x + 1 = 0,$$
it is a fixed point of the map $\mathbf{Q}(\pi)\mathbf{A} (1, \mu/s_{\ell})$ on the Riemann sphere. 
It is then easy to deduce that $i$ is a fixed point of $\mathbf{V}$, and a direct computation shows that $\infty$ is also a fixed point. 
Since $\mathbf{V}$ has the form $z \mapsto (az+ b)/(cz+ d)$ for some $a, b, c, d \in \mathbb{R}$, it is equal to the identity when it is viewed as a transformation of the Riemann sphere. Hence,  $\mathbf{V}$ is a translation by a multiple of $2 \pi$ when it is viewed as a transformation on $\mathbb{R}$. 
Since the real and the imaginary parts of $\rho_{\ell}$ are nonnegative, we have $\operatorname{Arg} (\rho_{\ell} )\in [0, \pi/2]$, and then 
$$\pi \lstar \mathbf{Q} ( -2 ( \pi - \operatorname{Arg} (\rho_{\ell}))) \in [-  \pi, 0] .$$
Since for all $a \in \mathbb{R}_+^*$, $b \in \mathbb{R}$, the maps $ x \mapsto x \lstar \mathbf{A}(a,b) $ are increasing and fix the odd multiples of $\pi$, we get 
$$  \pi \lstar \mathbf{Q} ( -2 ( \pi - \operatorname{Arg} (\rho_{\ell}))) \mathbf{A}^{-1} (\Im(\rho_{\ell})^{-1}, - \Re (\rho_{\ell})) \in [- \pi, \pi].$$
On the other hand, a direct computation gives 
$$ \infty .  \mathbf{Q} ( -2 ( \pi - \operatorname{Arg} (\rho_{\ell}))) \mathbf{A}^{-1} (\Im(\rho_{\ell})^{-1}, - \Re (\rho_{\ell}))  = 
0$$
and then necessarily
$$  \pi \lstar \mathbf{Q} ( -2 ( \pi - \operatorname{Arg} (\rho_{\ell}))) \mathbf{A}^{-1} (\Im(\rho_{\ell})^{-1}, - \Re (\rho_{\ell})) = 0,$$
which implies 
$$ \pi \lstar \mathbf{Q} ( -2 ( \pi - \operatorname{Arg} (\rho_{\ell}))) \mathbf{A}^{-1} (\Im(\rho_{\ell})^{-1}, - \Re (\rho_{\ell}))  \mathbf{Q}(\pi)\mathbf{A} (1, \mu/s_{\ell})  \mathbf{A} (\Im(\rho_{\ell})^{-1}, - \Re (\rho_{\ell})) = \pi.$$
Hence, $\mathbf{V}$ induces the identity map on $\mathbb{R}$. 
From \eqref{deltaxx}, we deduce 
$$  \Delta \varphi_{\ell, \Lambda, \mu}   =  y \lstar  \mathbf{S}_{\ell, \Lambda, \mu} - y.$$
On the other hand, a direct computation shows that 
$$ x \lstar  \mathbf{S}_{\ell, \Lambda, \mu} -  x= 0$$
when $x = \pi$, which implies
$$\Delta \varphi_{\ell, \Lambda, \mu}   = y \lstar  \mathbf{S}_{\ell, \Lambda, \mu} - y  -   x \lstar  \mathbf{S}_{\ell, \Lambda, \mu} + x$$
where $$e^{ix} = -1, \quad e^{iy} =   e^{i \varphi_{\ell, \Lambda, \mu}} \overline{ \eta_{\ell}},$$
i.e. 
$$\Delta \varphi_{\ell, \Lambda, \mu}
= \operatorname{ash} (\mathbf{S}_{\ell, \Lambda, \mu},
  -1, e^{i \varphi_{\ell, \Lambda, \mu}} \overline{ \eta_{\ell}}).$$
\end{proof}

 As in \cite{VV}, equation (50), we now introduce the following parameter: 
$$\lambda := 2 \sqrt{n_0}(\Lambda - \mu).$$
Our proof of Theorem \ref{boundvariance2017} is based on the following key estimates:
\begin{proposition} \label{estimatephi}
Let  $\mu \geq 0$,  $\Lambda \in \mathbb{R}$ such that 
$|\lambda| \leq n_0^{1/10}$, and $|\lambda| \leq 1$ if
$n_0 \leq n^{5/6}$. Then, for 
$$\ell := \max(0, \lceil n_0 - \mu^{2/3} - 1 \rceil ),$$
 which implies that $0 \leq \ell < n_0$,  the following holds: 
\begin{equation}
\sum_{j = 0}^{\ell - 1} \operatorname{Arg} (\rho_j)
= \frac{n}{2} \int_{(\mu/\sqrt{n}) \wedge 2}^{2} \sqrt{4 - x^2} dx + O(1).
\label{semicircle}
\end{equation}
\begin{equation}
 \mathbb{E} [(\varphi_{\ell, \mu, \mu})^2 ] = O(\log(2+n_0)),
 \label{boundphi1}
 \end{equation}
\begin{equation}
\mathbb{E} [(\varphi_{\ell,\Lambda, \mu} -\varphi_{\ell,\mu, \mu} - \lambda )^2 ] = O(\log(2 +|\lambda|)), \label{boundphi2}
\end{equation}
\begin{equation}
\mathbb{E} [(\tphh_{\ell, \mu} + 2 \pi (n-\ell))^2] = O(1). \label{boundtphh}
\end{equation}
Here, the implicit constant depends only on $\beta$.
\end{proposition}

The proof of  Proposition \ref{estimatephi}, which is very technical, is postponed to Section \ref{estphi}. 
We now prove that Proposition \ref{estimatephi} implies Theorem \ref{boundvariance2017}. 
\begin{proof} 
 We can assume $0 \leq \Lambda_1 < \Lambda_2$: the case $ \Lambda_1 < \Lambda_2 \leq 0$ is equivalent
by the symmetry of the distribution of the Beta Ensemble, and
for $\Lambda_1 < 0 <  \Lambda_2$, one can split the interval
into two pieces $(\Lambda_1,  0]$ and $(0, \Lambda_2)$.
Now, let us assume $\Lambda_2 = \infty$.
Then, for $\mu = \Lambda_1$ and $\ell$ satisfying Proposition \ref{estimatephi}, we get the following estimate, by applying the Minkowski inequality to a big telescopic sum: 
\begin{align} & 2 \pi \left(\mathbb{E} [(N_n(\Lambda_1, \infty) - N_{sc}(\Lambda_1, \infty))^2]\right)^{1/2} \nonumber
\\ & = 2 \pi \left(\mathbb{E} [( N_n(-\infty, \Lambda_1) - n + N_{sc}(\Lambda_1, \infty))^2]\right)^{1/2} \nonumber \\  & \leq \left(\mathbb{E} [(2 \pi  N_n(-\infty, \Lambda_1) - (\hat{\varphi}_{\ell, \mu} -
\tphh_{\ell, \mu}))^2]\right)^{1/2}
\label{abc1} \\ & \, +   \left( \mathbb{E} [(
 \hat{\varphi}_{\ell, \mu} - \varphi_{\ell, \mu, \mu} -  2 \sum_{j = 0}^{\ell - 1} (  \pi - \operatorname{Arg} (\rho_j)) )^2] \right)^{1/2}
\label{abc2} \\ & +  \left( \mathbb{E} [ ( - 2 \sum_{j = 0}^{\ell - 1} \operatorname{Arg} (\rho_j) + n \int_{\mu/\sqrt{n}}^{\infty} 
\sqrt{(4 - x^2)_+}  \,  dx )^2 ]\right)^{1/2} \label{abc3} 
\\ & \,  + \left( \mathbb{E} [(\tphh_{\ell, \mu} + 2\pi(n-\ell))^2] \right)^{1/2} \label{abc4} \\ & \, +   \left( \mathbb{E} [(\varphi_{\ell, \mu, \mu})^2] \right)^{1/2}. \label{triangleL2inequality}
\end{align}
By Proposition \ref{numbereigenvalues}, the  term \eqref{abc1}
is $O(1)$. By the definition of $\varphi_{\ell, \mu, \mu}$ and
the fact that
 $\mathbf{A} (\Im(\rho_{\ell})^{-1}, - \Re (\rho_{\ell}))$
 does not change the argument by more than $2 \pi$, the
term  \eqref{abc2} is also $O(1)$. Moreover, by
 Proposition \ref{estimatephi}, the terms \eqref{abc3} and \eqref{abc4} are $O(1)$, whereas the term  \eqref{triangleL2inequality} is
 $O( \sqrt{\log (2 + n_0)})$. Hence,
 \begin{equation}
 \mathbb{E} [(N_n(-\Lambda_1, \infty) - N_{sc}(\Lambda_1, \infty))^2] = O( \log (2 + n_0)), \label{qwertyuiopasdfghjkl}
 \end{equation}
which gives the theorem in the case $\Lambda_2  = \infty$.

Let us now suppose that $\Lambda_2 < \infty$ and
$\Lambda_2 - \Lambda_1 \geq \frac{1}{2 \sqrt{n}}(n_0(\Lambda_1))^{1/10}$.
Subtracting the estimates  \eqref{qwertyuiopasdfghjkl} for the intervals
$( \Lambda_1, \infty)$ and $( \Lambda_2, \infty)$ gives
the following:
$$\mathbb{E} [(N_n(\Lambda_1, \Lambda_2) - N_{sc}(\Lambda_1, \Lambda_2))^2]  = O( \log (2 + n_0(\Lambda_2)) +
\log (2 + n_0(\Lambda_1))) = O(\log (2 + n_0(\Lambda_1))),$$
since $n_0(\Lambda_1) \geq n_0(\Lambda_2)$. Now,
$\sqrt{n} (\Lambda_2 - \Lambda_1)
\geq \frac{1}{2} (n_0(\Lambda_1))^{1/10}$,
hence,
$$\log (2 + n_0(\Lambda_1)) = O( \log ( 2+ (\sqrt{n}(\Lambda_2 - \Lambda_1) ) \wedge n)),$$
which proves the theorem also in this case.

The remaining case is when $\Lambda_2 - \Lambda_1 \leq
 \frac{1}{2 \sqrt{n}}(n_0(\Lambda_1))^{1/10}$. Taking $\mu = \Lambda_1$ and $\Lambda = \Lambda_2$ gives $|\lambda| \leq \sqrt{n_0/n}( n_0^{1/10}) \leq n_0^{1/10}$, and for
 $n_0 \leq n^{5/6}$, $$|\lambda| \leq n_0^{3/5}n^{-1/2}
 \leq n^{(3/5)(5/6)} n^{-1/2} = 1.$$
  Moreover, we have the big telescopic sum:
\begin{align*}2 \pi( N_n(\Lambda_1, \Lambda_2) - N_{sc}(\Lambda_1, \Lambda_2))& = (2 \pi N_n(-\infty, \Lambda_2) - (\hat{\varphi}_{\ell, \Lambda} - \tphh_{\ell, \Lambda}))
 \\ & +  \left(
 \hat{\varphi}_{\ell, \Lambda} - \varphi_{\ell, \Lambda, \mu} - 2 \sum_{j = 0}^{\ell - 1} (\pi - \operatorname{Arg} (\rho_j)) \right)
\\ & -    \left(
 \hat{\varphi}_{\ell, \mu} - \varphi_{\ell, \mu, \mu} - 2 \sum_{j = 0}^{\ell - 1} (\pi - \operatorname{Arg} (\rho_j)) \right)
 \\ & + ( \tphh_{\ell, \mu} + 2 \pi (n-\ell)) 
 - (2 \pi N_n(-\infty, \Lambda_1) - (\hat{\varphi}_{\ell, \mu} - \tphh_{\ell, \mu}))
 \\ &  -( \tphh_{\ell, \Lambda} + 2 \pi (n-\ell))    +
 ( \varphi_{\ell, \Lambda, \mu} - \varphi_{\ell, \mu, \mu}
-  \lambda) + ( \lambda -  2 \pi N_{sc}(\Lambda_1, \Lambda_2)).
\end{align*}
 Bounding the $L^2$ norm as in \eqref{triangleL2inequality},
 we deduce, from all the estimates of Proposition
 \ref{estimatephi}, that
 $$\mathbb{E} [(N_n(\Lambda_1, \Lambda_2) - N_{sc}(\Lambda_1, \Lambda_2))^2] = O(\log ( 2 + |\lambda|) + ( \lambda -
 2 \pi N_{sc}(\Lambda_1, \Lambda_2))^2).$$
 Now,
 $$\log ( 2+ |\lambda|) \leq \log ( 2 + 2 \sqrt{n_0}
 (\Lambda_2 - \Lambda_1) \wedge n_0^{1/10} )
 \leq 2  \log ( 2 +  \sqrt{n}
 (\Lambda_2 - \Lambda_1) \wedge n),$$
 hence, it is sufficient to check that
 $$\lambda -
 2 \pi N_{sc}(\Lambda_1, \Lambda_2) = O(1).$$
One has the upper bound:
 \begin{align*}2 \pi N_{sc}(\Lambda_1, \Lambda_2)
 & \leq n \, \frac{\Lambda_2 - \Lambda_1}{\sqrt{n}} \,
 \sqrt{(4 - \Lambda_1^2/n)_+} = 2 (\Lambda_2 - \Lambda_1)
 \sqrt{(n- \mu^2/4)_+} \\ & \leq 2 (\Lambda_2 - \Lambda_1) \sqrt{n_0 + 1} = \lambda ( 1 + n_0^{-1})^{1/2}
 =  \lambda + O(n_0^{-9/10}) = \lambda + O(1).
 \end{align*}

 If $n_0 \leq n^{5/6}$, we have $|\lambda| \leq 1$, and
 then $2 \pi N_{sc}(\Lambda_1, \Lambda_2) = O(1)$,
 which gives the desired bound. We can now suppose
 $n_0 \geq n^{5/6}$. If $n \geq 10$ (the case $n \leq 9$ of the
 theorem is trivial), one deduces
 $$n - \frac{\Lambda_1^2}{4} - \frac{1}{2} \geq n^{5/6}$$
 and
 $$\Lambda_1 \leq 2 \sqrt{ n - n^{5/6} - \frac{1}{2}}
 \leq 2 \sqrt{n} - n^{1/3}.$$
 Hence,
 $$\Lambda_2 \leq 2 \sqrt{n} - n^{1/3} + \frac{\lambda}{2
 \sqrt{n_0}} \leq 2 \sqrt{n} - n^{1/3} + \frac{1}{2} n_0^{-2/5} \leq 2 \sqrt{n}.$$
 In other words, $(\Lambda_1/\sqrt{n}, \Lambda_2/\sqrt{n})$
 is included in the support of the semicircle law. Using
  the inequality $\sqrt{a-b} \geq \sqrt{a} - \sqrt{b}$ for
  $0 \leq b \leq a$, we get
 \begin{align*}2 \pi N_{sc}(\Lambda_1, \Lambda_2)
& \geq n \, \frac{\Lambda_2 - \Lambda_1}{\sqrt{n}} \,
 \sqrt{4 - \Lambda_2^2/n} \\ & \geq
  n \, \frac{\Lambda_2 - \Lambda_1}{\sqrt{n}} \,
 \left(\sqrt{4 - \Lambda_1^2/n} - \sqrt{(\Lambda_1^2 - \Lambda_2^2)/n}\right).
 \end{align*}
Now,
\begin{align*}
n \, \frac{\Lambda_2 - \Lambda_1}{\sqrt{n}}
\sqrt{4 - \Lambda_1^2/n} & \geq 2 (\Lambda_2 - \Lambda_1)
\sqrt{n- \Lambda_1^2/4} \geq 2   (\Lambda_2 - \Lambda_1)
\sqrt{n_0 - 1} =  \lambda \sqrt{ 1 - n_0^{-1}}
\\ & =  \lambda + O(n_0^{-9/10}),
\end{align*}
and
\begin{align*}n \, \frac{\Lambda_2 - \Lambda_1}{\sqrt{n}}
\sqrt{( \Lambda_2^2 - \Lambda_1^2)/n}
& = (\Lambda_2 - \Lambda_1)^{3/2} (\Lambda_2 + \Lambda_1)^{1/2} \\ & \leq  2^{-3/2}n_0^{-3/4} \lambda^{3/2} (4
\sqrt{n})^{1/2}.
\end{align*}
Since by assumption, $n_0 \geq n^{5/6}$, we deduce that
$$n \, \frac{\Lambda_2 - \Lambda_1}{\sqrt{n}}
\sqrt{( \Lambda_1^2 - \Lambda_2^2)/n}
\leq n_0^{-3/4}n_0^{3/20}n^{1/4} = n_0^{-3/5} n^{1/4}
 \leq n^{-(3/5)(5/6)}n^{1/4} = n^{-1/4} \leq 1,$$
 which completes the proof of the theorem.
 \end{proof} 
\section{Proof of Proposition \ref{estimatephi} }  \label{estphi} 
 From now, we suppose that the assumptions of Proposition \ref{estimatephi} are satisfied and that $\ell$ is equal to the positive part of $\lceil n_0 - \mu^{2/3} - 1\rceil$. Moreover, 
all the implicit constants in the estimates are allowed to depend only on $\beta$.

The present section is divided into five subsections. In the first subsection, we state two results adapted from \cite{VV}, which are needed in our proof of Proposition \ref{estimatephi}. 
In each of the four last subsections, we show one of the four estimates of the proposition. 
\subsection{Results adapted from Valk\'o and Vir\'ag  \cite{VV}}
The first result we need is the following: 
\begin{proposition}
\label{VV22}
We assume that $\ell \geq 1$, and then $n_0 > 1$, which implies 
$n_0 = n - \mu^2/4 - 1/2$.  
For $0 \leq j \leq \ell - 1$,  $\lambda \in [-n_0^{1/10}, n_0^{1/10}]$, let us define
$$\Delta \varphi_{j, \lambda} := \varphi_{j + 1, \Lambda , \mu} -  \varphi_{j, \Lambda , \mu},$$
where $$\Lambda = \mu + \frac{\lambda}{2 \sqrt{n_0}}.$$ Then, for $k = n_0 - j$, and $t = j/n_0$, we have 
\begin{equation} 
\mathbb{E} [\Delta \varphi_{j, \lambda} | \mathcal{F}_{j}]
= \frac{\lambda}{ 2 \sqrt{k n_0}} +  \frac{ b_0(t) }{n_0} + \frac{osc_{1,j}}{n_0} + O(k^{-3/2}) = O(k^{-1} + |\lambda| (k n_0)^{-1/2}) = O \left(k^{-1} + k^{-1/2} n_0^{-2/5} \right), \label{formuladeltaphi}
\end{equation}
where
$\mathcal{F}_{j}$ denotes the
$\sigma$-algebra generated by $X_0, \dots, X_{j-1}, Y_0, \dots, Y_{j-1}$, and where 
$$b_0(t) = \frac{\Im (\rho(t)^2)}{2 \beta (1-t)} -
\frac{\Re (\rho'(t))}{\Im (\rho(t))},$$
$$\rho(t):=
 \sqrt{\frac{\mu^2/4}{\mu^2/4 + n_0(1-t)}}
 + i \, \sqrt{\frac{n_0(1-t)}{\mu^2/4 + n_0(1-t)}},$$
 $$osc_{1,j} = \Re ((-v_0(t) + \lambda/(2 \sqrt{1-t}) - i q(t)/2)e^{-i \varphi_{j, \Lambda, \mu}} \eta_j ) - \Re (i q(t) e^{-2 i \varphi_{j, \Lambda, \mu}} \eta^2_j)/4,$$
 $$v_0(t) = \frac{\rho'(t)}{\Im(\rho (t))}, \; \;
q(t) = \frac{2 (1 + \rho(t)^2)}{\beta (1-t)}.$$
Moreover, 
\begin{equation} \mathbb{E} [ (\Delta  \varphi_{j, \lambda} )^d ] = O(k^{-d/2}) \label{moment234} 
\end{equation} 
for $d \in \{2,3,4\}$. 
All these estimates are uniform in $\lambda \in [-n_0^{1/10}, n_0^{1/10}]$.
\end{proposition}
This result is  (after correcting a sign error  in the last term of $osc_{1,j}$, which has no impact in our proof) a part of Proposition 22 of \cite{VV}, slightly modified in order to handle the uniformity in $\lambda$ and to get an estimate of the fourth moment of $\Delta  \varphi_{j, \lambda}$. 
\begin{proof}
By Proposition \ref{ash}, we have 
 $$\Delta \varphi_{j, \Lambda, \mu}
  = \varphi_{j+1 , \Lambda, \mu} -\varphi_{j, \Lambda, \mu} = \operatorname{ash} (\mathbf{S}_{\ell, \Lambda, \mu},
  -1, e^{i \varphi_{j, \Lambda, \mu}} \overline{ \eta_{j}}),$$
and then 
$$\Delta \varphi_{j, \lambda}
 = \operatorname{ash} (\mathbf{S}_{\ell, \Lambda, \mu},
  -1, e^{i \varphi_{j, \Lambda, \mu}} \overline{ \eta_{j}}),$$
for $\Lambda = \mu + \lambda/(2 \sqrt{n_0})$. 
Now, if we extend $\mathbf{S}_{\ell, \Lambda, \mu}$ meromorphically to a transformation of the Riemann sphere or the Poincar\'e half-plane $\{z \in \mathbb{C}, \Im(z) > 0\}$, we have 
$$ Z_{j, \lambda} := i. \mathbf{S}^{-1}_{\ell, \Lambda, \mu} - i = v_{j, \lambda} + V_{j} $$
where $ i. \mathbf{S}^{-1}_{\ell, \Lambda, \mu}$ denotes the image of $i$ by $\mathbf{S}^{-1}_{\ell, \Lambda, \mu}$, and 
$$v_{j, \lambda} = - \frac{\lambda}{2 n_0 \hat{s} (t)}  + \frac{\rho_{j+1} - \rho_j}{ \Im (\rho_j)}, \quad  V_{j} = \frac{X_j + \rho_{j+1} Y_j}{ \sqrt{n_0} \hat{s}(t)},$$
$$\hat{s} (t) = \sqrt{1 -t} = \sqrt{k/n_0}.$$
These formulas  correspond to equations (62), (63), (64) is \cite{VV} and can be proven by direct computation, using the definition of $\mathbf{S}^{-1}_{\ell, \Lambda, \mu}$ given in Proposition \ref{ash}, and the fact that $s_j \Im(\rho_j) = \sqrt{k}$, because 
$$s_j = \sqrt{n - (1/2) -j },  \quad \Im(\rho_j) = \sqrt{\frac{n_0 - j}{ \mu^2/4 + n_0 - j}} = \sqrt{\frac{n_0 - j}{ n  - j - 1/2}}, \quad k  = n_0 - j.$$
We then use Lemma 16 of \cite{VV} in order to estimate $\Delta \varphi_{j, \lambda}$. We get the equation (72) of \cite{VV}, which gives
$$\Delta \varphi_{j, \lambda} = - \Re Z + \frac{\Im (Z^2)}{4} - \Re (  \bar{z} \eta Z + i \bar{z} \eta Z^2/2 + i \bar{z}^2 \eta^2 Z^2/4) + O(|Z|^3),$$
for $Z = Z_{j, \lambda}$, $\eta = \eta_j$ and $z = e^{i \varphi_{j, \Lambda, \mu}}$. 
The only randomness in $Z_{j, \lambda}$ comes from $X_j$ and $Y_j$, and then  $Z_{j, \lambda}$ is independent of $\mathcal{F}_j$, whereas 
$\eta_j$ is deterministic and $e^{i \varphi_{j, \Lambda, \mu}}$ is  $\mathcal{F}_j$-measurable. We deduce 
\begin{equation}
\mathbb{E} [ \Delta \varphi_{j, \lambda} | \mathcal{F}_j] = - \Re (\mathbb{E} [  Z]) + \frac{ \Im( \mathbb{E} [ Z^2])}{4} - \Re (   \bar{z} \eta \mathbb{E}[ Z] + i \bar{z} \eta \mathbb{E} [  Z^2]/2 + i \bar{z}^2 \eta^2 \mathbb{E} [  Z^2]/4 ) + O ( \mathbb{E} [ |Z|^3]). \label{equationDeltaZ} 
\end{equation}
We have by equation (65) of \cite{VV}, 
 $$v_{j,0} = \frac{v_0(t)}{n_0} + O(k^{-2})$$
and the estimate 
$$|v_0(t)| = O(n_0/k).$$
Since
$$v_{j, \lambda} = - \frac{\lambda}{2 n_0 \hat{s} (t)} + v_{j,0},$$
we deduce 
$$v_{j, \lambda} = - \frac{\lambda}{2 n_0 \hat{s} (t)}  + \frac{v_0(t)}{n_0} + O(k^{-2}) = - \frac{\lambda}{2 n_0 \hat{s} (t)} + O(1/k).$$
Notice that the last part of  \cite{VV}, equation (65) does not hold uniformly in $\lambda$.  However, since we assume $|\lambda| \leq n_0^{1/10}$, we have
$$v_{j, \lambda}  =- \frac{\lambda}{2 n_0 \sqrt{k/n_0}} + O(1/k) = O\left( \frac{n_0^{1/10}}{\sqrt{k n_0}} \right) + O(1/k) 
= O (n_0^{-0.4} k^{-0.5}) + O(1/k) = O(k^{-0.9}).$$

On the other hand, \cite{VV}, equation (66) gives the estimates: 
$$\mathbb{E} [ V_j] = O((n-j)^{-3/2} k^{-1/2}), \quad \mathbb{E} [ V_j^2 ]= \frac{q(t)}{n_0} + O((n-j)^{-1/2} k^{-3/2}), \quad \mathbb{E} [ |V_j|^d ] = O(k^{-d/2})$$
for $d \in \{3,4\}$. 
Since $n-j \geq n_0 - j =k$, the error terms in $\mathbb{E} [ V_j]$ and  $\mathbb{E} [ V_j^2 ]$ are dominated by $k^{-2}$. Hence, 
$$\mathbb{E} [ Z ] =v_{j, \lambda} +  \mathbb{E} [V_j] = - \frac{\lambda}{2 n_0 \hat{s} (t)}  + \frac{v_0(t)}{n_0} + O(k^{-2}) 
= \frac{1}{n_0} \left( -  \frac{\lambda}{2\sqrt{1-t} } + v_0(t) \right) +  O(k^{-2}).$$
Similarly, 
$$\mathbb{E} [ Z^2 ]  =v^2 _{j, \lambda}   + 2 v_{j, \lambda}\mathbb{E} [V_j] +\mathbb{E} [V^2_j]
= O(k^{-1.8}) + O(k^{-0.9}) O(k^{-2}) +   \frac{q(t)}{n_0} + O(k^{-2}) =     \frac{q(t)}{n_0} + O(k^{-1.8}).$$
Hence, 
\begin{align*} 
 - \Re (\mathbb{E} [  Z]) + \frac{ \Im( \mathbb{E} [ Z^2])}{4} & =  \frac{\lambda} {2 \sqrt{k n_0}} + \frac{1}{n_0} \left( - \Re (v_0(t))  + \Im ( q(t))/4 \right) + O(k^{-1.8})
\\ & =  \frac{\lambda} {2 \sqrt{k n_0}} + \frac{1}{n_0} \left( - \frac{\Re (\rho'(t))}{ \Im(\rho(t))}   +  \frac{ \Im (\rho(t)^2)}{2 \beta (1-t)} \right) + O(k^{-1.8})
\\ & =  \frac{\lambda} {2 \sqrt{k n_0}} + \frac{ b_0(t)}{ n_0} +  O(k^{-1.8}). 
\end{align*} 
Similarly, we have 
$$- \Re (   \bar{z} \eta \mathbb{E}[ Z] + i \bar{z} \eta \mathbb{E} [  Z^2]/2 + i \bar{z}^2 \eta^2 \mathbb{E} [  Z^2]/4 )  =\frac{osc_{1,j}}{n_0} + O(k^{-1.8}),$$
and then  by \eqref{equationDeltaZ}: 
$$\mathbb{E} [ \Delta \varphi_{j, \lambda} | \mathcal{F}_j] =  \frac{\lambda} {2 \sqrt{k n_0}} + \frac{ b_0(t)}{ n_0} +\frac{osc_{1,j}}{n_0} +  O ( \mathbb{E} [ |Z|^3]) +   O(k^{-1.8}). $$
Now, 
$$\mathbb{E} [ |Z|^3 ] = O(  |v_{j, \lambda}|^3 + \mathbb{E}[ | V_{j} |^3] ) = O ( k^{-2.7} + k^{-3/2}) = O(k^{-3/2}),$$
and then we get the first part of  \eqref{formuladeltaphi}.
Moreover, equation (65) of \cite{VV} gives the estimate 
$v_0 (t)  = O(n_0/k)$,  whereas we have from the definition of $q$:
$$q(t) = O \left(\frac{1}{1-t} \right) = O( n_0/k)$$
and then 
$$b_0(t) = \Im (q(t))/4 - \Re (v_0(t)) = O(n_0/k).$$
Injecting these estimates in the first part of  \eqref{formuladeltaphi} gives the second part. The third part is immediately deduced from the second part by using the assumption $|\lambda| \leq n_0^{1/10}$. 
Moreover, using Lemma 16 of \cite{VV}, more precisely the last estimate of equation (37) in this lemma, we get 
$$\Delta \varphi_{j, \lambda}
 = \operatorname{ash} (\mathbf{S}_{\ell, \Lambda, \mu},
  -1, e^{i \varphi_{j, \lambda}} \overline{ \eta_{j}}) = O(|Z|),$$
and then for $d \in \{2,3,4\}$, 
$$\mathbb{E} [ |\Delta \varphi_{j, \lambda} |^d] = O (  |v_{j, \lambda}|^d + \mathbb{E}[ | V_{j} |^d] ) = O(k^{-0.9 d} + k^{-d/2}) = O(k^{-d/2}),$$
which gives \eqref{moment234}. 
\end{proof} 
The other result we need is the following: 
\begin{lemma} \label{VV37} 
Let $\ell_0$ and $\ell_1$ be integers such that $0 \leq \ell_0 < \ell_1 \leq n_0$. Then, for $\ell_0 \leq j \leq \ell_1$, we have
$$\sum_{m= \ell_0}^j \eta_m = O \left( \mu k^{-1/2} + 1 \right), \quad \sum_{m= \ell_0}^j \eta^2_m  =  O \left( \mu k^{-1/2} + \mu^{-1} k_0^{1/2} \right).$$
Moreover, for any complex numbers $(g_j)_{\ell_0 \leq j \leq \ell_1}$, 
$$\left| \Re \left( \sum_{j = \ell_0}^{\ell_1}  g_j \eta_j \right) \right| = O \left( ( \mu k_1^{-1/2} + 1 ) |g_{\ell_1}| + \sum_{j= \ell_0}^{\ell_1-1} ( \mu k^{-1/2} + 1 ) | g_{j+1} - g_j| \right)$$
and
$$ \left| \Re \left( \sum_{j = \ell_0}^{\ell_1}  g_j \eta^2_j \right) \right| = O \left( ( \mu k_1^{-1/2} +  \mu^{-1} k_0^{1/2} ) |g_{\ell_1}| + \sum_{j= \ell_0}^{\ell_1-1} ( \mu k^{-1/2} +  \mu^{-1} k_0^{1/2} ) | g_{j+1} - g_j| \right),$$
Here, 
$$k_0 := n_0 - \ell_0, \quad  k_1 := n_0 - \ell_1, \quad  k := n_0 - j.$$
\end{lemma} 

\begin{proof}
From the definition of $\rho_j$, one checks that the argument of $\rho_j$ is decreasing in $j$ and stays in the interval $[0, \pi/2]$. Moreover, 
we have  
$$ \tan ( \operatorname{Arg} (\rho_j) ) =  2 \mu^{-1} k^{1/2}, \quad \tan ( (\pi/2) -   \operatorname{Arg} (\rho_j) ) =  \mu k^{-1/2}/2,$$
which implies that 
$$( \operatorname{Arg} (\rho_j) )^{-1} = O (1 +  \mu k^{-1/2}), \quad ( (\pi/2) -   \operatorname{Arg} (\rho_j) )^{-1} = O (1 +  \mu^{-1} k^{1/2}).$$
Using the definition of $\eta_m$ in terms of the $\rho_j$'s and applying Lemma 36 of \cite{VV}, we deduce the first part of the lemma. 
The second part of the lemma exactly corresponds to Lemma 37 of \cite{VV}. 
\end{proof}

\subsection{Proof of \eqref{semicircle}}
\begin{proof} 
In this subsection, we prove the estimate: 
$$\sum_{j = 0}^{\ell - 1} \operatorname{Arg} (\rho_j)
= \frac{n}{2} \int_{(\mu/\sqrt{n}) \wedge 2}^{2} \sqrt{4 - x^2} dx + O(1).$$
 If $n \geq 2$ (which can be assumed) and $n_0 = 1$, then $\mu \geq \sqrt{4n-6}$, and
 necessarily $\ell = 0$. The left-hand side of
 \eqref{semicircle} is empty: hence, it is sufficient to check
 $$n \int_{\sqrt{4 - (6/n)}}^{2} \sqrt{4 - x^2} \, dx = O(1),$$
 which is straightforward.
 If $n_0 > 1$, then $n_0 = n - \mu^2/4 - (1/2)$ and
 $$\operatorname{Arg} (\rho_j) = \arccos
\left( \sqrt{ \frac{\mu^2}{4n - 4j - 2}} \right).$$
 This quantity is decreasing in $j$.  Hence,
$$\sum_{j= 1}^{\ell - 1} \operatorname{Arg} (\rho_j)
\leq \int_0^{\ell}  \arccos
\left( \sqrt{ \frac{\mu^2}{4n - 4x - 2}} \right) \, dx
\leq  \sum_{j= 0}^{\ell-1} \operatorname{Arg} (\rho_j).$$
Since all the arguments here are $O(1)$, \eqref{semicircle}
is equivalent to
$$\int_{1/2}^{ 1/2 + \ell}  \arccos
\left( \frac{\mu}{2\sqrt{n-y}} \right) \, dy
=  \frac{n}{2} \int_{\mu/\sqrt{n}}^2 \sqrt{4 - x^2} dx + O(1)
$$
(note that $\mu/\sqrt{n} \leq 2$, since $n_0 > 1$),
or to
$$\int_{0}^{ \ell}  \arccos
\left( \frac{\mu}{2\sqrt{n-y}} \right) \, dy
=   \frac{n}{2} \int_{\mu/\sqrt{n}}^2 \sqrt{4 - x^2} dx + O(1).
$$
Now, it is not difficult to check the equality:
$$\int_{0}^{ n_0 + 1/2}\arccos
\left( \frac{\mu}{2\sqrt{n-y}} \right) \, dy
  = \int_{0}^{n - (\mu^2/4)}\arccos
\left( \frac{\mu}{2\sqrt{n-y}} \right) \, dy
  = \frac{n}{2} \int_{\mu/\sqrt{n}}^2 \sqrt{4 - x^2} \, dx.$$
Hence, \eqref{semicircle} is satisfied if and only if
$$\int_{\ell}^{ n_0 }\arccos
\left( \frac{\mu}{2\sqrt{n-y}} \right) \, dy = O(1).$$
Now,
\begin{align*}
\int_{\ell}^{ n_0 }\arccos
\left( \frac{\mu}{2\sqrt{n-y}} \right) \, dy
& \leq (n_0 - \ell) \arccos
\left( \frac{\mu}{2\sqrt{n-\ell}} \right)
\\ & = (n_0 - \ell) \arctan \left(\frac{2}{\mu}
\sqrt{n - \ell - (\mu^2/4)} \right)
\\ & = (n_0 - \ell) \arctan \left(\frac{2}{\mu}
\sqrt{n_0 - \ell + (1/2)} \right)
\\ & = O( (n_0 - \ell+1) \wedge \mu^{-1} (n_0- \ell + 1)^{3/2}),
\end{align*}
which is $O(1)$ since $n_0 - \ell \leq 1 + \mu^{2/3}$ by the definition of $\ell$. This completes the proof of \eqref{semicircle}.
\end{proof}
\subsection{Proof of \eqref{boundphi1}}
\begin{proof}
 In this subsection, we prove the estimate: 
$$\mathbb{E} [(\varphi_{\ell, \mu, \mu})^2 ] = O(\log(2+n_0)).$$
 If $\ell = 0$, it is trivial, so we can assume that
we are in the situation where $\ell \geq 1$. 
Then, we  decompose the phase $\varphi_{\ell, \mu, \mu}$ as the sum of 
a martingale part and a predictible part, with respect to the filtration $(\mathcal{F}_{r})_{0 \leq r \leq \ell}$, where 
we recall that $\mathcal{F}_{r}$ is the
$\sigma$-algebra generated by $X_0, \dots, X_{r-1}, Y_0, \dots, Y_{r-1}$.  More precisely, we define, for 
 $0 \leq r \leq \ell - 1$:
$$\Delta \varphi_r := \Delta \varphi_{r,0} =  \varphi_{r+1, \mu, \mu}
- \varphi_{r, \mu, \mu}$$
and for $0 \leq r \leq \ell$,
$$M_{r} := \varphi_{r, \mu, \mu} - \sum_{j = 0}^{r-1}
\mathbb{E} [\Delta \varphi_j | \mathcal{F}_{j}].$$
For $0 \leq r \leq \ell$, we then have the decomposition: 
$$\varphi_{r, \mu, \mu} = M_{r} +   \sum_{j = 0}^{r-1}
\mathbb{E} [\Delta \varphi_j | \mathcal{F}_{j}],$$
where $(M_{r})_{0 \leq r \leq \ell}$ is a martingale
with respect to $(\mathcal{F}_{r})_{0 \leq r \leq \ell}$, and the sum in $j$ is predictable with respect to the same filtration. 

The bound \eqref{boundphi1} will then be deduced from the following estimates: 
 \begin{equation}
\mathbb{E} [M_{\ell}^2]  = O(\log (2 + n_0)) \label{estimatemartingale11} 
\end{equation} 
and 
\begin{equation}
\mathbb{E}\left[\left(\sum_{j = 0}^{\ell-1}
\mathbb{E} [\Delta \varphi_j | \mathcal{F}_{j}] \right)^2
\right] = O (1). \label{boundsumconditionalexpectation}
\end{equation}
 
The estimate  \eqref{estimatemartingale11} can then be quickly obtained as follows. One has $M_0 = \pi$ and for $0 \leq r \leq
\ell -1$,
$$\mathbb{E} [(M_{r+1} - M_r)^2]
= \mathbb{E} [(\Delta \varphi_r -
\mathbb{E} [\Delta \varphi_r |\mathcal{F}_r])^2]
\leq  \mathbb{E} [(\Delta \varphi_r)^2] = O(1/k),
$$
where $k = n_0 - r$,  the last estimate corresponding to \eqref{moment234} in Proposition \ref{VV22}, for $d = 2$. 
Hence,  
$$\mathbb{E} [M_{\ell}^2] = \pi^2 + \sum_{ r= 0}^{\ell-1}
\mathbb{E} [(M_{r+1} - M_r)^2]
= O\left( \pi^2 + \sum_{k= n_0 - \ell+1}^{n_0} \frac{1}{k} \right)
 = O(\log (2 + n_0)),$$
which proves \eqref{estimatemartingale11}. 
In the last expression, notice that $n_0$ is not necessarily an integer: in this case, we use the convention
 $$\sum_{a = b}^c y_a := \sum_{a = 0}^{c-b} y_{a + b}$$
 if $c-b$ is an integer, even if $b$ and $c$ are not integers themselves. This convention will be implicitly used several times in the sequel of the paper.

It now remains to prove  \eqref{boundsumconditionalexpectation}.
By  \eqref{formuladeltaphi} in Proposition \ref{VV22}, it is enough to show the following estimates:

\begin{equation}\left( \sum_{j=0}^{\ell - 1} b_0(j/n_0) \right)^2
 =  O(n_0^2) \label{boundbb}
 \end{equation}
 and
 \begin{equation}
 \mathbb{E} \left[ \left(\sum_{j=0}^{\ell - 1}
 osc_{1,j} \right)^2 \right] = O(n_0^2 ),
 \label{boundosc}
 \end{equation}
when $\lambda  = 0$. 
 One has:
 \begin{align*}b_0(j/n_0) & =
 \frac{\mu \sqrt{n_0 -j}}{2 \beta (1 - (j/n_0))
 (\mu^2/4 + n_0 - j)} - \frac{n_0 \mu/4}
 {\sqrt{n_0 -j} (\mu^2/4 + n_0 - j)}
 \\ & = O \left(\frac{n_0 \mu}
 {\sqrt{n_0 -j} (\mu^2 + n_0 - j)}\right),
 \end{align*}
 and then
 $$\sum_{j = 0}^{\ell-1} b_0(j/n_0)
  = O \left( n_0 \mu \sum_{k=1}^{\infty}
  \frac{1}{(\mu^2 + k) \sqrt{k}} \right),$$
  where
  $$\sum_{k=1}^{\infty} \frac{1}{(\mu^2 + k) \sqrt{k}}
  \leq \left(\frac{1}{1+\mu^2} \sum_{1  \leq k \leq \mu^2}
  \frac{1}{\sqrt{k}} \right) + \sum_{k \geq 1 \vee \mu^2} \frac{1}{k^{3/2}} = O(1/\mu).$$
  Hence, we get:
  $$\sum_{j = 0}^{\ell-1} b_0(j/n_0)
  = O(n_0),$$
  which gives \eqref{boundbb}.
It remains to prove  \eqref{boundosc}. One has: 
  $$\rho'(t) = \frac{ n_0 \mu/4 - i \mu^2 \sqrt{n_0}
  /(8 \sqrt{1 - t})}{ (\mu^2/4 + n_0(1-t))^{3/2}}$$
  and
  $$\rho''(t) =
  \frac{3 n_0^2 \mu/8 - i \left( \mu^2 n_0^{3/2}/4\sqrt{1-t} + \mu^4 \sqrt{n_0}/64(1-t)^{3/2} \right)}{(\mu^2/4
  + n_0(1-t))^{5/2}},$$
  which gives
  $$v_0(t) = \frac{ n_0 \mu/4 - i \mu^2 \sqrt{n_0}
  /(8 \sqrt{1 - t})}{ \sqrt{n_0(1-t)} (\mu^2/4 + n_0(1-t))},$$
  \begin{align*}
   v_0'(t) &  = \frac{\rho''(t)}{\Im(\rho(t))} - \frac{\rho'(t) \Im(\rho'(t))}{\Im^2(\rho(t))}
   \\ & =  \frac{3 n_0^2 \mu/8 - i \left( \mu^2 n_0^{3/2}/4\sqrt{1-t} + \mu^4 \sqrt{n_0}/64(1-t)^{3/2} \right)}{
   \sqrt{n_0(1-t)} (\mu^2/4
  + n_0(1-t))^{2}}
   \\ & - \frac{ (n_0 \mu/4 - i \mu^2 \sqrt{n_0}
  /(8 \sqrt{1 - t})) (-\mu^2 \sqrt{n_0}
  /(8 \sqrt{1 - t})}{(n_0(1-t))(\mu^2/4 + n_0(1-t))^{2}},
  \end{align*}
  and then
  $$|v_0(t)| = O  \left( \frac{1}{1-t} \right),$$
  $$|v_0'(t)| = O  \left( \frac{1}{(1-t)^2} \right).$$
Similarly, one gets, from $|\rho(t)| = 1$ and the estimate $|\rho'(t)| = O(1/(1-t))$:
$$|q(t)| = O(1/(1-t))$$
and $$|q'(t)| = O(1/(1-t)^2).$$

  On the other hand,
  \begin{equation}
\sum_{j=0}^{\ell - 1}
 osc_{1,j} = \sum_{j=0}^{\ell - 1} \Re(e_{1,j} \eta_j) +
\frac{1}{4} \sum_{j=0}^{\ell - 1} \Re(e_{2,j} \eta^2_j) \label{oscillatory2019} 
\end{equation}
 where
 $$e_{1,j} =  (-v_0(t) - i q(t)/2)e^{-i \varphi_{j, \mu, \mu}}$$
 and
 $$e_{2,j} = - i q(t) e^{-2 i \varphi_{j, \mu, \mu}}.$$
 Now, let us define $\tilde{\ell}$ as follows: 
 $$\tilde{\ell} := \ell =  \left( \lceil n_0 - \mu^{2/3} - 1 \rceil \right)_+ $$
if $\mu \leq 1$, and 
$$\tilde{\ell} := \left( \lceil n_0 - \mu^2 - 1 \rceil \right)_+$$
if  $\mu > 1$. In any case, $\tilde{\ell} \leq \ell$. Moreover,
 $$ \left| \sum_{j=0}^{\tilde{\ell} -1} \Re(e_{2,j} \eta^2_j ) \right|
 \leq \sum_{j=0}^{\tilde{\ell} - 1} |e_{2,j}| \leq \sum_{j=0}^{\tilde{\ell} - 1} |q(j/n_0)|
 = O \left( n_0 \sum_{j=0}^{\tilde{\ell}-1} \frac{ |1 + \rho^2 (j/n_0) |}{ n_0 - j} \right).$$
 Now, since $\Re (\rho(j/n_0))$ and $\Im (\rho(j/n_0))$ are nonnegative, we have, by taking the argument
 in $[0, \pi/2]$,
 $$
 0 \leq \frac{\pi}{2} - \operatorname{Arg} (\rho (j/n_0)) \leq \frac{\pi}{2} \Re (\rho(j/n_0)) \leq \frac{\mu}{\sqrt{n_0 - j}},
 $$
 and since $\rho^2 (j/n_0) =  - e^{-2 i (\pi/2 -  \operatorname{Arg}  (\rho (j/n_0)))}$,
 $$|1 + \rho^2 (j/n_0)| = | 1 - e^{-2 i (\pi/2 -  \operatorname{Arg}  (\rho (j/n_0)))}|
 \leq 2 \left| \frac{\pi}{2} -  \operatorname{Arg}  (\rho (j/n_0)) \right| \leq \frac{2 \mu}{\sqrt{n_0 - j}}.$$
 Hence,
 $$ \left| \sum_{j=0}^{\tilde{\ell} - 1} \Re(e_{2,j} \eta^2_j ) \right|
 = O \left( n_0 \sum_{j=0}^{\tilde{\ell}-1} \frac{\mu}{(n_0 - j)^{3/2}} \right)
  = O \left(n_0 \sum_{k = n_0 - \tilde{\ell} + 1}^{\infty}  \frac{\mu}{k^{3/2}} \right)
  = O \left( \frac{n_0 \mu}{ \sqrt{n_0 - \tilde{\ell} + 1}} \right).$$
 Now, if $\tilde{\ell} > 0$ (otherwise the sum in $j$ just above is empty),
 $n_0 - \tilde{\ell} + 1  \geq n_0 - (n_0 - \mu^2) + 1 \geq \mu^2$, and then
 $$  \left( \sum_{j=0}^{\tilde{\ell} - 1} \Re(e_{2,j} \eta^2_j ) \right)^2 = O(n_0^2).$$
Subtracting this bound from  \eqref{oscillatory2019}, we deduce that \eqref{boundosc} is proven if we show the following estimate, 
for $d = 1$ and for $d = 2$:
 \begin{equation}
\mathbb{E} \left[ \left| \sum_{j=\ell'}^{\ell - 1} e_{d,j} \eta^d_j \right|^2 \right]
 = O(n_0^2 ), \label{sumellprime}
 \end{equation}
 where $\ell' := 0$ if $d = 1$, and $\ell' := \tilde{\ell}$ if $d = 2$.
 We know that $\ell' \leq \ell$, and the result is obvious for $\ell'=  \ell$ (the sum is empty), so we can
 assume $\ell' \leq \ell-1$. From the definition of $\tilde{\ell}$, we deduce that for $d = 2$, this assumption
 implies $\mu > 1$ and $\ell' = \tilde{\ell} \geq n_0 - \mu^2 - 1$.

Now, following the proof of Lemma 37 of \cite{VV}, we use
 partial summation, in order to write, for $d \in \{1, 2\}$, if $\ell \geq 1$,
 $$\sum_{j=\ell'}^{\ell - 1} e_{d,j} \eta^d_j
 = F_{d, \ell - 1} e_{d,\ell-1} + \sum_{j=\ell'}^{\ell-2} F_{d,j} (e_{d,j} - e_{d,j+1}),$$
 where
 $$F_{d,j} = \sum_{m = \ell'}^{j}  \eta^d_m,$$
 and where the second sum is empty for $\ell' = \ell-1$.
 Hence,
 $$ \sum_{j=\ell'}^{\ell - 1} e_{d,j} \eta^d_j =
 A_{d} + d i B_{d}$$
 where
 \begin{equation}
A_d
 = F_{d, \ell - 1} e_{d,\ell-1} + \sum_{j=\ell'}^{\ell-2} F_{d,j} (e_{d,j} - e_{d,j+1}
  - d i e_{d,j} \Delta \varphi_j)
\label{Ad}
\end{equation}
  and
  \begin{equation}
B_d =   \sum_{j=\ell'}^{\ell-2} F_{d,j} e_{d,j} \Delta \varphi_j. \label{Bd} 
\end{equation}
  In order to prove \eqref{boundosc}, it is then sufficient to show that
  the expectations of $|A_d|^2$ and $|B_d|^2$ are dominated by $n_0^2 $.

 From the estimates of $v_0$ and $q$ given above, we get $|e_{d,j}| = O(n_0/k)$ where $k = n_0 - j$. Moreover,
 if $$q_d(t) := (-v_0(t) - i q(t)/2) \mathds{1}_{d = 1} - i q(t) \mathds{1}_{d = 2},$$
 we get
 \begin{align*}
& e_{d,j} - e_{d,j+1} - d i e_{d,j} \Delta \varphi_j \\ &
 =  e^{-d i \varphi_{j, \mu, \mu}} [q_d(j/n_0) (1 - e^{-di \Delta \varphi_j }
 - d i \Delta \varphi_j)  +  (q_d(j/n_0) - q_d((j+1)/n_0) ) e^{-di \Delta \varphi_j } ].
 \end{align*}
 and then
 $$ |e_{d,j} - e_{d,j+1} + d i e_{d,j} \Delta \varphi_j|
 \leq  (d^2/2) |q_d(j/n_0)| |\Delta \varphi_j|^2 + \int_{j/n_0}^{(j+1)/n_0} |q_d'(u)| du. $$
Since $j \leq \ell - 2 \leq n_0-2$, we have $2(1 - u) \geq  1 - t= 1 - (j/n_0)$ for all
$u$ on the interval of integration. The previous estimates on $q$, $q'$, $v_0$, $v'_0$
give
$$ |e_{d,j} - e_{d,j+1} + d i e_{d,j} \Delta \varphi_j|
= O\left( |\Delta \varphi_j|^2 / (1-t) + 1/(n_0 (1-t)^2) \right)
 = O\left( |\Delta \varphi_j|^2 n_0/k + n_0/k^2 \right).
$$
By Lemma \ref{VV37}, we have:  
 $$ |F_{d,j}| = O \left( \mu k^{-1/2} + \mu^{-1} (n_0 - \ell')^{1/2} \mathds{1}_{d = 2}  + 1
 \right),$$
where $k = n_0 - j$. 
For $d = 2$, $\ell' = \tilde{\ell}$, and then the modulus of $A_d$ is dominated by
\begin{align}  & \frac{n_0}{n_0 - \ell + 1} \,\left(  1 +  \mathds{1}_{d = 2} \sqrt{n_0 -
\tilde{\ell}}/\mu  + \mu/\sqrt{n_0 - \ell
+ 1}    \right)
 \nonumber \\ & + \sum_{j = \ell'}^{\ell-2} \left( 1+ \mathds{1}_{d = 2}
 \sqrt{n_0 - \tilde{\ell} }/\mu + \mu/\sqrt{k}   \right)
 \left(n_0/k^2 +  |\Delta \varphi_j|^2\, n_0/k  \right). \label{boundAd}
 \end{align}
Now, since we assume $\ell \geq 1$ and then $\ell \leq n_0 - \mu^{2/3}$, we
have $n_0 - \ell + 1  \geq n_0 - (n_0 - \mu^{2/3}) + 1 =  1+ \mu^{2/3}$, and then the first
term is smaller than or equal to
$$ \frac{n_0}{1+\mu^{2/3}} + \mathds{1}_{d = 2}
\frac{n_0 (n_0 - \tilde{\ell})^{1/2}}{\mu^{5/3}} + n_0 \leq 2 n_0 +
\mathds{1}_{d = 2} \frac{n_0 (n_0 - \tilde{\ell})^{1/2}}{\mu^{5/3}}.
$$
Now, for $d = 2$, the assumptions we have made imply $\mu > 1$, and
$n_0 - \tilde{\ell} \leq \mu^2 + 1$: we deduce that the square of the first term of \eqref{boundAd} is
dominated by $n_0^2$.

In \eqref{boundAd}, let us now bound the part on the second term which does not involve $\Delta \varphi_j$.
This sum is smaller than of equal to
\begin{align*}
& \sum_{k=n_0 - \ell + 2}^{n_0}  \left( 1+ \mathds{1}_{d = 2}
\sqrt{n_0 - \tilde{\ell}}/\mu + \mu/\sqrt{k}   \right)  (n_0/k^2)
\\ & \leq \frac{n_0}{n_0 - \ell + 1} + \frac{n_0 \mu}{(n_0 - \ell + 1)^{3/2}} +
\frac{n_0 (n_0 - \tilde{\ell})^{1/2}}{ \mu (n_0 - \ell + 1) } \mathds{1}_{d = 2}
\\ & \leq  2 n_0 + \frac{n_0 (n_0 - \tilde{\ell})^{1/2}}{ \mu (n_0 - \ell + 1) } \mathds{1}_{d = 2}.
\end{align*}
For $d = 2$, we have assumed $\mu > 1$ and $n_0 - \tilde{\ell} \leq 1 + \mu^2$, and then
we deduce a bound of order $n_0^2$ for the square of the part of \eqref{boundAd} which does not involve
$\Delta \varphi_j$.

In order to bound the term involving  $\Delta \varphi_j$,  we use \eqref{moment234} in Proposition \ref{VV22}, which gives 
$$\mathbb{E} [(\Delta \varphi_j)^4] = O(k^{-2}).$$
We deduce that for all $j \in \{\ell', \dots, \ell -2\}$, in the term of the sum in \eqref{boundAd}
which is indexed by $j$, the $L^2$ norm of the part depending on $\Delta \varphi_j$ is dominated by the
part which does not depend on $\Delta \varphi_j$. Hence, the $L^2$ norm of the sum is dominated by its  part
not depending on $\Delta \varphi_j$, and then the expectation of $|A_d|^2$ is dominated by $n_0^2$.

Let us now bound the expectation of $|B_d|^2$. This expectation is dominated by
$$ \mathbb{E} \left[  \left( \sum_{j=\ell'}^{\ell-2} F_{d,j} e_{d,j} \mathbb{E} [\Delta \varphi_j | \mathcal{F}_j]
\right)^2 \right] +
\mathbb{E} \left[  \left( \sum_{j=\ell'}^{\ell-2} F_{d,j} e_{d,j} (M_{j+1} - M_j)
\right)^2 \right],$$
where we recall that
$$M_{r} := \varphi_{r, \mu, \mu} - \sum_{j = 0}^{r-1}
\mathbb{E} [\Delta \varphi_j | \mathcal{F}_{j}].$$
Inside the square in the second term, we have a sum of martingale increments. Hence, the second term is
equal to
$$ \mathbb{E} \left[  \sum_{j=\ell'}^{\ell-2} \left( F_{d,j} e_{d,j} (M_{j+1} - M_j)
\right)^2 \right].$$
Now, the second moment of $M_{j+1} - M_j$ is dominated by $1/k$, $e_{d,j}$ is dominated by $n_0/k$: hence, using
the previous bound on $F_{d,j}$, the expression just above is dominated by
\begin{equation} \sum_{j= \ell'}^{\ell - 2} \frac{n_0^2}{k^3}
\left(\mu^2 k^{-1} + \mu^{-2} (n_0 - \ell') \mathds{1}_{d = 2}  + 1 \right), \label{sumBd1}
\end{equation}
and then by
$$\frac{n_0^2 \mu^2}{ (n_0 - \ell + 1)^3} + \frac{n_0^2}{(n_0 - \ell + 1)^2} + \frac{n_0^2 (n_0 - \ell')
}{(n_0 - \ell + 1)^2 \mu^2} \mathds{1}_{d = 2}.$$
From $n_0 - \ell + 1 \geq 1 + \mu^{2/3}$, we deduce that the two first terms are dominated by $n_0^2$.
If $d = 2$, we have assumed $\mu > 1$ and then $n_0 - \ell' \leq 1+ \mu^2$, which again dominates the
corresponding term by $n_0^2$.

In order to get a satisfactory $L^2$ bound for $B_d$, it then remains to show
$$ \mathbb{E} \left[  \left( \sum_{j=\ell'}^{\ell-2} F_{d,j} e_{d,j} \mathbb{E} [\Delta \varphi_j | \mathcal{F}_j]
\right)^2 \right] = O (n_0^2).$$
Now, from \eqref{formuladeltaphi} in Proposition \ref{VV22}, we have $\mathbb{E} [\Delta \varphi_j | \mathcal{F}_j] = O(1/k)$ (recall that $\Lambda = \mu$ here, and then $\lambda = 0$).
Hence, the sum inside the square is dominated by
\begin{equation} \sum_{j= \ell'}^{\ell - 2} \frac{n_0}{k^2}
\left(\mu k^{-1/2} + \mu^{-1} (n_0 - \ell')^{1/2} \mathds{1}_{d = 2}  + 1 \right), \label{sumBd2}
\end{equation}
and then by
$$ \frac{n_0 \mu}{ (n_0 - \ell + 1)^{3/2}} + \frac{n_0}{n_0 - \ell + 1} + \frac{n_0 (n_0 - \ell')^{1/2}
}{(n_0 - \ell + 1) \mu} \mathds{1}_{d = 2},$$
which is dominated by $n_0$.
Hence
$$ \mathbb{E} \left[  \left( \sum_{j=\ell'}^{\ell-2} F_{d,j} e_{d,j} \mathbb{E} [\Delta \varphi_j | \mathcal{F}_j]
\right)^2 \right] = O (n_0^2),$$
 which completes the proof of \eqref{boundosc}, and then the proof of \eqref{boundphi1}.
 \end{proof}
\subsection{Proof of  \eqref{boundphi2}}
\begin{proof}
  Let us now show the estimate: 
$$\mathbb{E} [(\varphi_{\ell,\Lambda, \mu} -\varphi_{\ell,\mu, \mu} - \lambda )^2 ] = O(\log(2 +|\lambda|)).$$
 If $\ell = 0$, we have $\varphi_{\ell, \Lambda, \mu} - \varphi_{\ell, \mu, \mu} = O(1)$. 
Moreover, $n_0 \leq 1 + \mu^{2/3}$ and then, for  $\mu \leq 10 \sqrt{n}$, $n_0 \leq n^{5/6}$ if $n$ is large enough, which implies $|\lambda| \leq 1$, and for $\mu > 10 \sqrt{n}$, 
$n_0 = 1$ and then $|\lambda| \leq n_0^{1/10} = 1$. Hence, we have \eqref{boundphi2} in this case.

 From now, we can then assume $\ell \geq 1$. 
  We define 
$$\ell^*  = \ell \wedge \lfloor 1 + n_0 \left(1 - (1+|\lambda|)^{-2} \right) \rfloor.$$
which implies that $\ell^* \geq 1$.  Let us denote, for
$0 \leq r \leq \ell - 1$:
$$\Delta \psi_r := [\varphi_{r+1, \Lambda, \mu} - \varphi_{r+1, \mu, \mu}]
- [\varphi_{r, \Lambda, \mu} - \varphi_{r, \mu, \mu}] - \frac{\lambda}{2 \sqrt{(n_0 - r) n_0}}.$$
We can then decompose $\Delta \psi_r $ as the sum of the increments of a martingale, and the increments of a predictable process. 
If for $0 \leq r \leq \ell$, we define:
$$N_{r} := [\varphi_{r, \Lambda, \mu} - \varphi_{r, \mu, \mu}] - \sum_{j = 0}^{r-1}
\mathbb{E} [\Delta \psi_j | \mathcal{F}_{j}] -  \sum_{j=0}^{r-1} \frac{\lambda}{2 \sqrt{(n_0-j) n_0}},$$
then the sequence $(N_{r})_{0 \leq r \leq \ell}$ is a martingale
with respect to the filtration $(\mathcal{F}_{r})_{0 \leq r \leq \ell}$, since
$$N_{r+1} - N_r = \Delta \psi_r -
\mathbb{E} [\Delta \psi_r |\mathcal{F}_r].$$
 One has $N_0 = O(1)$ and for $0 \leq r \leq
\ell -1$,
\begin{align*} \mathbb{E} [(N_{r+1} - N_r)^2]
& = \mathbb{E} [(\Delta \psi_r -
\mathbb{E} [\Delta \psi_r |\mathcal{F}_r])^2]
\leq  \mathbb{E}  \left[\left(\Delta \psi_r +  \frac{\lambda}{2 \sqrt{(n_0 - r) n_0}} \right)^2 \right]
\\ & =
 \mathbb{E} [\left( [\varphi_{r+1, \Lambda, \mu} - \varphi_{r+1, \mu, \mu}]
- [\varphi_{r, \Lambda, \mu} - \varphi_{r, \mu, \mu}] \right)^2].
\end{align*}
From \eqref{moment234} in Proposition \ref{VV22}, which holds uniformly in $|\lambda| \leq n_0^{1/10}$, we deduce: 
\begin{align*} \mathbb{E} [N_{\ell^*}^2] = O(1) + \sum_{ r= 0}^{\ell^*-1}
\mathbb{E} [(N_{r+1} - N_r)^2]
& = O\left(  1 +\sum_{k= n_0 - \ell^*+1}^{n_0} \frac{1}{k} \right)
\\ & = O\left( 1+  \log(n_0) - \log (n_0 - \ell^* + 1)  \right)
\\ & = O \left( 1 + \log(n_0) - \log [n_0/(1+|\lambda|)^2] \right)
= O \left( \log (2 + |\lambda|) \right).
\end{align*}
 In order to prove \eqref{boundphi2}, it is then sufficient to show that
 \begin{equation}
\mathbb{E}\left[\left(\sum_{j = 0}^{\ell^*-1}
\mathbb{E} [\Delta \psi_j | \mathcal{F}_{j}] \right)^2
\right] = O (1) \label{bounddeltapsi}
 \end{equation}
 and
 \begin{equation}
 \mathbb{E} \left[ \left(  (\varphi_{\ell, \Lambda, \mu} - \varphi_{\ell, \mu, \mu}) -
 (\varphi_{\ell^*, \Lambda, \mu} - \varphi_{\ell^*, \mu, \mu}) - \left( \lambda
 -  \sum_{j=0}^{\ell^*-1} \frac{\lambda}{2 \sqrt{(n_0-j) n_0}} \right) \right)^2 \right] = O(1),
 \label{boundremainingterm}
 \end{equation}
since 
\begin{align*}
& \varphi_{\ell,\Lambda, \mu} -\varphi_{\ell,\mu, \mu} - \lambda \\ & =  (\varphi_{\ell, \Lambda, \mu} - \varphi_{\ell, \mu, \mu}) -
 (\varphi_{\ell^*, \Lambda, \mu} - \varphi_{\ell^*, \mu, \mu}) - \left( \lambda
 -  \sum_{j=0}^{\ell^*-1} \frac{\lambda}{2 \sqrt{(n_0-j) n_0}} \right) + N_{\ell^*}  + \sum_{j=0}^{\ell^* - 1} \mathbb{E} [\Delta \psi_j | \mathcal{F}_{j}].
\end{align*}
The sequel of the subsection is devoted to the proof of  \eqref{bounddeltapsi} and  \eqref{boundremainingterm}.

 {\bf  Proof of \eqref{bounddeltapsi}:} This estimate is a consequence of 
 \begin{equation}
\mathbb{E}\left[\left(\sum_{j = 0}^{\ell^*-1}
\mathbb{E} [\Delta \varphi_j | \mathcal{F}_{j}] \right)^2 \label{estimatewithlstar1}
\right] = O (1),
\end{equation}
and
\begin{equation} \mathbb{E}\left[\left(\sum_{j = 0}^{\ell^*-1}
\mathbb{E} \left [\Delta \varphi_{j,\lambda} -  \frac{\lambda}{2 \sqrt{(n_0-j) n_0}}  \big| \mathcal{F}_{j} \right] \right)^2
\right] = O (1), \label{estimatewithlstar2}
\end{equation}
where we recall that $\Delta \varphi_{j, \lambda} =  \varphi_{j+1, \Lambda, \mu} -  \varphi_{j, \Lambda, \mu}$. Indeed, 
$$ \Delta \psi_j  = \Delta \varphi_{j,\lambda} -   \Delta \varphi_{j} -  \frac{\lambda}{2 \sqrt{(n_0-j) n_0}}.$$
The estimate \eqref{estimatewithlstar1} is the same as \eqref{boundsumconditionalexpectation}, except that $\ell$ is replaced by
$\ell^*$, i.e. there are less terms in the sum. One then checks that the proof of \eqref{boundsumconditionalexpectation} still works here if we replace
$\ell, \tilde{\ell}, \ell'$ by their infimums $\ell^*, \tilde{\ell^*}, (\ell^*)'$ with $\ell^*$.

For the second estimate \eqref{estimatewithlstar2}, we use Proposition \ref{VV22}, which gives: 
\begin{equation}
\mathbb{E} \left [\Delta \varphi_{j,\lambda} -  \frac{\lambda}{2 \sqrt{(n_0-j) n_0}}  \big| \mathcal{F}_{j} \right]
= \frac{ b_0(t) }{n_0} + \frac{osc_{1,j}}{n_0} + O(k^{-3/2}). \label{uniformlambdaestimate}
\end{equation}
uniformly in $|\lambda| \leq n_0^{1/10}$.
 We can then prove \eqref{estimatewithlstar2} in a similar way as
 \eqref{boundsumconditionalexpectation}, except that we should take into account the term $\lambda/(2 \sqrt{1-t})$ in the expression of $osc_{1,j}$. 
 Since $v_0(t) = O(1/(1-t))$ and $v'_0(t) = O(1/(1-t)^2)$, we deduce that 
$$ v_{0} (t) - \frac{\lambda}{2\sqrt{1-t}}   = O \left(\frac{1}{1-t} + \frac{|\lambda|}{\sqrt{1-t}} \right)$$ and
$$\frac{d}{dt} \left( v_{0} (t) - \frac{\lambda}{2\sqrt{1-t}} \right) = O \left( \frac{1}{(1-t)^2} 
 + \frac{|\lambda|}{(1-t)^{3/2}} \right).$$
The term depending on $\lambda$ multiplies all the estimates of the sums involving $e_{1,j}$ by
 $ 1 + |\lambda| \sqrt{k/n_0}$ (recall that $k = n_0 -j$ and $t = j/n_0$ in our computation). Hence, the proof of
  \eqref{boundsumconditionalexpectation} works without change if $|\lambda| \leq 1$. If $|\lambda| > 1$, we necessarily
  have $n_0 > n^{5/6}$ by assumption, which implies $n - \mu^2/4 - 1/2 > 1$,  $\mu = O(\sqrt{n})$, and
  then $\mu = O(\sqrt{n_0^{6/5}}) = O(n_0^{3/5})$.
  Moreover, we have $|\lambda| \leq n_0^{1/10}$, which implies that the estimates are at most multiplied by
  $1 + k^{1/2} n_0^{-2/5}$. The proof is then unchanged until quantities similar to $A_d$ and $B_d$ are 
introduced, as in \eqref{Ad} and \eqref{Bd}. The expectation of the squared modulus of these quantities should be dominated by $n_0^2$. 
The modulus of the analog of $A_d$ is bounded by the sum of the analog of  \eqref{boundAd} where $\ell$, $\tilde{\ell}$ and $\ell'$ are replaced by $\ell^*$, $\tilde{\ell}^*$ and $(\ell^*)'$, and an extra term due to the multiplication by   $1 + k^{1/2} n_0^{-2/5}$ in the estimates of terms involving $e_{1,j}$, which is equal to 
   \begin{align}  & \frac{n_0^{3/5}}{(n_0 - \ell^* + 1)^{1/2}} \,\left(  1 +  \mathds{1}_{d = 2} \sqrt{n_0 -
\tilde{\ell^*}}/\mu  + \mu/\sqrt{n_0 - \ell^*
+ 1}    \right)
 \nonumber \\ & + \sum_{j = (\ell^*)'}^{\ell^*-2} \left( 1+ \mathds{1}_{d = 2}
 \sqrt{n_0 - \tilde{\ell^*} }/\mu + \mu/\sqrt{k}   \right)
 \left(n_0^{3/5}/k^{3/2} +  |\Delta \varphi_{j, \lambda}|^2\, n_0^{3/5}/k^{1/2}  \right).
 \end{align}
Here,  we can assume that $(\ell^*)'  \leq \ell^*-1$ (otherwise the sum similar to
\eqref{sumellprime} we have to bound is empty), which implies that $(\ell^*)' = \ell'$, and then
$\tilde{\ell^*} =\tilde{\ell}$ for $d = 2$. We deduce that we can bound the modulus of the analog of $A_d$ by the sum of  \eqref{boundAd}  and 
the extra term 
 \begin{align}  & \frac{n_0^{3/5}}{(n_0 - \ell + 1)^{1/2}} \,\left(  1 +  \mathds{1}_{d = 2} \sqrt{n_0 -
\tilde{\ell}}/\mu  + \mu/\sqrt{n_0 - \ell
+ 1}    \right)
 \nonumber \\ & + \sum_{j = \ell'}^{\ell-2} \left( 1+ \mathds{1}_{d = 2}
 \sqrt{n_0 - \tilde{\ell} }/\mu + \mu/\sqrt{k}   \right)
 \left(n_0^{3/5}/k^{3/2} +  |\Delta \varphi_{j, \lambda}|^2\, n_0^{3/5}/k^{1/2}  \right), \label{boundmodifiedAd}
 \end{align}
where the stars have been removed. 
The term  \eqref{boundAd} has been already suitably estimated, so we can focus on \eqref{boundmodifiedAd}, whose squared modulus should have an expectation dominated by $n_0^2$. 
 Here, we assume $\ell \geq 1$, which implies $\ell \leq n_0 - \mu^{2/3}$,  we
have $n_0 - \ell + 1  \geq n_0 - (n_0 - \mu^{2/3}) + 1 =  1+ \mu^{2/3}$, and then the first
term of  \eqref{boundmodifiedAd} is smaller than or equal to
$$ \frac{n^{3/5}_0}{(1+\mu^{2/3})^{1/2}} + \mathds{1}_{d = 2}
\frac{n_0^{3/5} (n_0 - \tilde{\ell})^{1/2}}{\mu^{4/3}} + n_0^{3/5}\mu^{1/3} \leq n_0^{3/5}(1+ \mu^{1/3}) +
\mathds{1}_{d = 2} \frac{n_0^{3/5} (n_0 - \tilde{\ell})^{1/2}}{\mu^{4/3}}.
$$
Now, for $d = 2$, the assumptions we have made imply $\mu > 1$, and
$n_0 - \tilde{\ell} \leq \mu^2 + 1$, moreover $\mu$ is dominated by $n_0^{3/5}$. We deduce that the first term
of \eqref{boundmodifiedAd}
is dominated by $n_0^{4/5}$, which is more than enough for our purpose.

In \eqref{boundmodifiedAd}, let us now bound the part of the second term which does not involve $\Delta
\varphi_{j,\lambda}$.
This sum is smaller than or equal to
\begin{align*}
& \sum_{k=n_0 - \ell + 2}^{n_0}  \left( 1+ \mathds{1}_{d = 2}
\sqrt{n_0 - \tilde{\ell}}/\mu + \mu/\sqrt{k}   \right)  (n_0^{3/5}/k^{3/2})
\\ & \leq \frac{n_0^{3/5}}{(n_0 - \ell + 1)^{1/2}} + \frac{n_0^{3/5} \mu}{n_0 - \ell + 1} +
\frac{n_0^{3/5} (n_0 - \tilde{\ell})^{1/2}}{ \mu (n_0 - \ell + 1)^{1/2} } \mathds{1}_{d = 2}
\\ & \leq  \frac{n_0^{3/5}}{(1+\mu^{2/3})^{1/2}} + n_0^{3/5} \mu^{1/3} +
\frac{n_0^{3/5} (n_0 - \tilde{\ell})^{1/2}}{ \mu^{4/3} } \mathds{1}_{d = 2}
\end{align*}
For $d = 2$, we have assumed $\mu > 1$ and $n_0 - \tilde{\ell} \leq 1 + \mu^2$, and then
we deduce again a bound of order $n_0^{4/5}$ for the part of \eqref{boundmodifiedAd} which does not involve
$\Delta \varphi_{j, \lambda}$.
Moreover, we have $ \mathbb{E} [(\Delta \varphi_{j, \lambda})^4] =  O(k^{-2})$
by \eqref{moment234} in Proposition \ref{VV22}, which implies that the part of  \eqref{boundmodifiedAd} depending on $\Delta \varphi_{j, \lambda}$ is dominated in
$L^2$ by the part not depending on $\Delta \varphi_{j, \lambda}$.
We have now proven a suitable bound for the analog of $A_d$. 
 
It remains to get a suitable $L^2$ bound for the quantity similar to $B_d$: see \eqref{Bd}. By 
\eqref{moment234} in Proposition \ref{VV22}, we have
$\mathbb{E} [(\Delta \varphi_{j, \lambda})^2] = O(1/k)$. Hence, the proof of the bound of $B_d$ can be adapted to the
present situation. The sums \eqref{sumBd1} and \eqref{sumBd2} should be modified. Since
$e_{d,j}$ is squared in the quantity estimated by \eqref{sumBd1}, the terms should be multipled by
a quantity of order $1 + k n_0^{-4/5}$, which gives an extra sum:
 $$ \sum_{j= \ell'}^{\ell - 2} \frac{n_0^{6/5}}{k^2}
\left(\mu^2 k^{-1} + \mu^{-2} (n_0 - \ell') \mathds{1}_{d = 2}  + 1 \right),$$
This sum is dominated by
$$\frac{n_0^{6/5} \mu^2}{ (n_0 - \ell + 1)^2} + \frac{n_0^{6/5}}{n_0 - \ell + 1} + \frac{n_0^{6/5} (n_0 - \ell')
}{(n_0 - \ell + 1) \mu^2} \mathds{1}_{d = 2}.$$
From $n_0 - \ell + 1 \geq 1 + \mu^{2/3}$, we deduce that the two first terms are dominated by $n_0^{6/5}
(1 + \mu^{2/3}) = O(n_0^{8/5})$.
If $d = 2$, we have assumed $\mu > 1$ and then $n_0 - \ell' \leq 1+ \mu^2$, which dominates the
corresponding term by $n_0^{6/5}$.
The terms in \eqref{sumBd2} should be multiplied by $1 + k^{1/2} n_0^{-2/5}$ because of the estimate of
$e_{d,j}$, but also multiplied a second time by $1 + k^{1/2} n_0^{-2/5}$ because the bound given by 
\eqref{formuladeltaphi} in Proposition \ref{VV22} is 
$O(k^{-1} (1 + k^{1/2}n_0^{-2/5}))$ instead of $O(k^{-1})$. We then get a factor
dominated again by $1 + k n_0^{-4/5}$, which gives an extra term:
$$ \sum_{j= \ell'}^{\ell - 2} \frac{n_0^{1/5}}{k}
\left(\mu k^{-1/2} + \mu^{-1} (n_0 - \ell')^{1/2} \mathds{1}_{d = 2}  + 1 \right),$$
 dominated by
$$ \frac{n_0^{1/5} \mu}{ (n_0 - \ell + 1)^{1/2}} + n_0^{1/5} \log (2 + n_0) + \frac{n_0^{1/5} \log(2+n_0)
(n_0 - \ell')^{1/2}}{\mu} \mathds{1}_{d = 2},$$
and then by $n_0^{3/5}$, which completes the proof of \eqref{bounddeltapsi}.

{\bf Proof of \eqref{boundremainingterm}: }For $a \geq 1$ and $b-a \geq 0$ integer, we have
$$\sum_{k=a}^{b} \frac{1}{\sqrt{k}} = 2 \left( \sqrt{b} - \sqrt{a}\right) + O\left( 1 \right),$$
and then
$$\sqrt{n_0} - \sum_{j=0}^{\ell^* -1} \frac{1}{2 \sqrt{n_0 - j}} =
\sqrt{n_0 - \ell^* + 1} +  O\left( 1 \right).$$
Now, if $\ell^* < \ell$, we get
$$ n_0 - \ell^* + 1 = n_0 - n_0 \left( 1 - \frac{1}{(1 + |\lambda|)^2} \right) + O(1)
= \frac{n_0}{(1 + |\lambda|)^2} + O(1),$$
and then
$$ \sqrt{n_0} - \sum_{j=0}^{\ell^* -1} \frac{1}{2 \sqrt{n_0 - j}}
= \frac{\sqrt{n_0}}{1 + |\lambda|} + O(1).$$
Multiplying by $\lambda/\sqrt{n_0}$ gives
$$ \lambda - \sum_{j=0}^{\ell^* -1} \frac{\lambda}{2 \sqrt{n_0(n_0 - j)}}
 =  \frac{\lambda}{1 + |\lambda|} + O(n_0^{1/10}/\sqrt{n_0}) = O(1).$$
 If $\ell^* = \ell$,
 $$ \sqrt{n_0} - \sum_{j=0}^{\ell^* -1} \frac{1}{2 \sqrt{n_0 - j}}
 = \sqrt{n_0 - \ell + 1} + O(1) = ( \mu^{2/3} \wedge n_0)^{1/2} + O(1),$$
 $$ \lambda - \sum_{j=0}^{\ell^* -1} \frac{\lambda}{2 \sqrt{n_0(n_0 - j)}}
 =  O( 1 + ( \lambda^2 \mu^{2/3} n_0^{-1} \wedge \lambda^2)^{1/2} ).$$
 If $|\lambda| \leq 1$, this gives a bound $O(1)$. Otherwise, we have $n_0 > n^{5/6}$ by
 assumption, and then $\mu= O(\sqrt{n}) = O(n_0^{3/5})$, which gives
 $$\lambda^2 \mu^{2/3} n_0^{-1} = O (n_0^{2/10} n_0^{2/5} n_0^{-1})$$
 and then again a bound $O(1)$.
 Hence,  \eqref{boundremainingterm} is proven if we show
 $$\mathbb{E} \left[ \left(  (\varphi_{\ell, \Lambda, \mu} - \varphi_{\ell, \mu, \mu}) -
 (\varphi_{\ell^*, \Lambda, \mu} - \varphi_{\ell^*, \mu, \mu})  \right)^2 \right] = O(1), $$
 i.e.
 \begin{equation}
  \mathbb{E} [(R_{\ell} - R_{\ell^*})^2 ] = O(1) \label{boundRell}
 \end{equation}
where
 $$R_j := \varphi_{j, \Lambda, \mu} - \varphi_{j, \mu, \mu}.$$
 Of course we can assume that $\ell^* < \ell$.
We now assume $\lambda \geq 0$, the case $\lambda \leq 0$ is similar by changing the suitable signs.
 Using \eqref{formuladeltaphi} in Proposition \ref{VV22}, 
 we get
 $$ \mathbb{E}[R_{j+1} - R_j |\mathcal{F}_j] = \frac{\lambda}{2\sqrt{k n_0}}
 + \frac{1}{n_0}  \widetilde{osc}_{1,j}  + O(k^{-3/2}),$$
 where
 \begin{align*}
\widetilde{osc}_{1,j} & := \Re ((-v_0(t) - i q(t)/2)(e^{-i \varphi_{j, \Lambda, \mu}} - e^{-i \varphi_{j, \mu, \mu}} )
 \eta_j )  \\ & - \frac{1}{4} \Re (i q(t) (e^{-2 i \varphi_{j, \Lambda, \mu}}
 - e^{-2 i \varphi_{j, \mu, \mu}}) \eta^2_j) + \frac{\lambda}{2 \sqrt{1-t}}
 \Re (e^{-i \varphi_{j, \Lambda, \mu}}\eta_j).
 \end{align*}
 (recall that $t = j/n_0$ and then $1 - t = k/n_0$).
 We deduce:
 $$\mathbb{E}[R_{j+1} - R_j |\mathcal{F}_j] = \frac{1}{n_0}  \widetilde{osc}'_{1,j}
 + O \left(k^{-3/2} + \frac{\lambda}{\sqrt{k n_0}} \right),$$
 where
 $$ \widetilde{osc}'_{1,j} := \Re ((-v_0(t) - i q(t)/2)(e^{-i \varphi_{j, \Lambda, \mu}}
 - e^{-i \varphi_{j, \mu, \mu}} )
 \eta_j )  - \frac{1}{4} \Re (i q(t) (e^{-2 i \varphi_{j, \Lambda, \mu}}
 - e^{-2 i \varphi_{j, \mu, \mu}}) \eta^2_j).$$
 Hence, for  $\ell^* \leq \ell_1 <  \ell_2 \leq \ell - 1$,
 $$\mathbb{E}[R_{\ell_2} - R_{\ell_1}|\mathcal{F}_{\ell_1}]
 = O \left( \sum_{j= \ell_1}^{\ell_2-1} k^{-3/2} + \frac{\lambda}{n_0^{1/2}}
 \sum_{j = \ell_1}^{\ell_2-1} k^{-1/2} \right)
 + \sum_{j = \ell_1}^{\ell_2-1} \Re( g_{1,j} \eta_j )
 + \frac{1}{4} \sum_{j = \ell_1}^{\ell_2-1} \Re( g_{2,j} \eta^2_j )
 ,$$
 where
 $$g_{1,j} = \frac{1}{n_0}(-v_0(t) - i q(t)/2) \mathbb{E} [ e^{-i \varphi_{j, \Lambda, \mu}}
 - e^{-i \varphi_{j, \mu, \mu}} | \mathcal{F}_{\ell_1}],$$
 $$g_{2,j} = - \frac{1}{n_0} i q(t) \mathbb{E} [(e^{-2 i \varphi_{j, \Lambda, \mu}}
 - e^{-2 i \varphi_{j, \mu, \mu}}) | \mathcal{F}_{\ell_1}],$$
 for $t = j/n_0$.
 Because of the choice of $\ell^*$, the maximal possible value of $k$ is at most $n_0/(1+ |\lambda|)^2$, which implies
 that the sum in the last $O$ is dominated by $1$, uniformly in $\ell_1$ and $\ell_2$:
 $$ \mathbb{E}[R_{\ell_2} - R_{\ell_1}|\mathcal{F}_{\ell_1}]
 =\sum_{j = \ell_1}^{\ell_2-1} \Re( g_{1,j} \eta_j )
 + \frac{1}{4} \sum_{j = \ell_1}^{\ell_2-1} \Re( g_{2,j} \eta^2_j ) + O(1).$$
 Now, we have the bound:
 $$ |g_{1,j}| \leq \frac{1}{n_0} |v_0(t) + i q(t)/2|  = O(1/k).$$
 Moreover,
 \begin{align*}
 & |g_{1,j+1} - g_{1,j}|
 \\ & \leq
 \frac{1}{n_0}  (|v_0((j+1)/n_0)|+ |q((j+1)/n_0)|) \left| \mathbb{E} [ ( e^{-i\varphi_{j+1, \Lambda, \mu}}
 - e^{-i\varphi_{j, \Lambda, \mu}} ) - ( e^{-i\varphi_{j+1, \mu, \mu}}
 - e^{-i\varphi_{j, \mu, \mu}} )  \, | \mathcal{F}_{\ell_1} ] \right|
 \\ & + \frac{1}{n_0} ( |v_0((j+1)/n_0) - v_0(j/n_0)| +
  |q((j+1)/n_0) - q(j/n_0)| ) | \mathbb{E} [ e^{-i \varphi_{j, \Lambda, \mu}}
 - e^{-i \varphi_{j, \mu, \mu}} | \mathcal{F}_{\ell_1}]|.
 \end{align*}
 We know that
 \begin{align*} & \mathbb{E} [ ( e^{-i\varphi_{j+1, \Lambda, \mu}}
 - e^{-i\varphi_{j, \Lambda, \mu}} ) - ( e^{-i\varphi_{j+1, \mu, \mu}}
 - e^{-i\varphi_{j, \mu, \mu}} ) | \,  \mathcal{F}_{\ell_1} ]
  \\ & =- i  \mathbb{E} [ ( \varphi_{j+1, \Lambda, \mu}
 - \varphi_{j, \Lambda, \mu}) e^{-i\varphi_{j, \Lambda, \mu}} -  (\varphi_{j+1, \mu, \mu}
 - \varphi_{j, \mu, \mu}) e^{-i\varphi_{j, \mu, \mu}}   | \,  \mathcal{F}_{\ell_1} ]
\\ &  + O \left( \mathbb{E} [ ( \varphi_{j+1, \Lambda, \mu}
 - \varphi_{j, \Lambda, \mu})^2 +   (\varphi_{j+1, \mu, \mu}
 - \varphi_{j, \mu, \mu} )^2 | \,  \mathcal{F}_{\ell_1} ]  \right).
 \end{align*}
 Using the estimates of Proposition \ref{VV22},  we deduce 
 $$ \mathbb{E} [ ( e^{-i\varphi_{j+1, \Lambda, \mu}}
 - e^{-i\varphi_{j, \Lambda, \mu}} ) - ( e^{-i\varphi_{j+1, \mu, \mu}}
 - e^{-i\varphi_{j, \mu, \mu}} ) | \,  \mathcal{F}_{\ell_1} ] = O  ( 1/k + \lambda/\sqrt{k n_0}).$$
 From the previous estimates of $v_0$, $q$ and their derivatives, this implies
 $$ |g_{1,j+1} - g_{1,j}| =  O \left( \frac{1}{k^2} + \frac{\lambda}{k^{3/2} n_0^{1/2}}  \right). $$
 Hence, from Lemma \ref{VV37}: 
 $$ \sum_{j = \ell_1}^{\ell_2-1} \Re( g_{1,j} \eta_j )
 = O \left( \frac{1}{n_0 - \ell_2 + 1} +  \frac{\mu}{(n_0 - \ell_2+1)^{3/2}} + \sum_{\ell_1 \leq j \leq \ell_2 - 2}
 (1 + \mu k^{-1/2})(1/k^2 + \lambda/(k^{3/2} n_0^{1/2}) )\right),$$
 and then
 $$  \sum_{j = \ell_1}^{\ell_2-1} \Re( g_{1,j} \eta_j )
 = O \left( \frac{1}{n_0 - \ell_2 + 1} +  \frac{\mu}{(n_0 - \ell_2+1)^{3/2}}
+  \frac{\lambda}{n_0^{1/2} (n_0 - \ell_2 + 1)^{1/2}} +  \frac{\mu \lambda}{n_0^{1/2}(n_0 - \ell_2+1)}
\right).$$
Now, since we assume $\ell \geq 1$, we have $\ell < n_0 - \mu^{2/3}$, and then
$$n_0 - \ell_2 + 1 \geq n_0 - \ell + 1 \geq 1 + \mu^{2/3}, \; n_0 \geq 1 + \mu^{2/3}.$$
The two first terms of the estimate above are bounded, and also the third since we assume
$\lambda \leq n_0^{1/10}$. If $\lambda \leq 1$, we immediately see that the last term is also bounded.
If $\lambda > 1$, we have made the extra assumption $n_0 > n^{5/6} \geq 1$, which implies, from the definition
of $n_0$, that $\mu = O (n^{1/2})$, and then $\mu = O (n_0^{3/5})$, and
$$ \frac{\mu \lambda}{n_0^{1/2}(n_0 - \ell_2+1)}
\leq \frac{\mu \lambda}{n_0^{1/2}\mu^{2/3}} = \frac{\mu^{1/3} \lambda}{n_0^{1/2}}
= O( n_0^{1/5} n_0^{1/10} n_0^{-1/2}) = O(1).$$
Then,
$$  \sum_{j = \ell_1}^{\ell_2-1} \Re( g_{1,j}\eta_j )  = O(1).$$
We have, for $|g_{2,j}|$ and $|g_{2,j+1} - g_{2,j}|$, the same estimates as for
$|g_{1,j}|$ and $|g_{1,j+1} - g_{1,j}|$, proven exactly in the same way.
We have previously defined $\tilde{\ell}$ as $\ell$ if $\mu \leq 1$, and as 
$( \lceil n_0 - \mu^2 - 1 \rceil )_+$ if $\mu > 1$. If we apply Lemma \ref{VV37} to
the sum $\sum_{\ell_1 \vee \tilde{\ell} \leq j \leq \ell_2-1} \Re( g_{2,j} \eta^2_j )$, we obtain
a similar estimate as we obtained
for the same sum with $g_{2,j}$ replaced by $g_{1,j}$. Indeed, if the sum is non-empty, necessarily
$\tilde{\ell} < \ell$, which implies $\mu > 1$ and $\tilde{\ell} \geq n_0 - \mu^2 - 1$, and then
the highest possible value $k_{\max}$ of $k$ involved in the sum is at most $\mu^2 + 1 < 2 \mu^2$,
which implies that in the estimates of Lemma \ref{VV37},   $\mu^{-1} k_{\max}^{1/2} = O(1)$, i.e.
the second estimate of the lemma is dominated by the first one.
Hence,
$$  \sum_{j = \ell_1 \vee \tilde{\ell}}^{\ell_2-1} \Re( g_{2,j} \eta^2_j )  = O(1).$$
Now,
 $$ \left| \sum_{\ell_1 \leq j \leq \ell_2 - 1, j < \tilde{\ell}} \Re(g_{2,j} \eta^2_j ) \right|
 \leq \sum_{j=0}^{\tilde{\ell} - 1} |g_{2,j}| \leq \frac{2}{n_0} \sum_{j=0}^{\tilde{\ell} - 1} |q(j/n_0)|
 = O \left( \sum_{j=0}^{\tilde{\ell}-1} \frac{ |1 + \rho^2 (j/n_0) |}{ n_0 - j} \right).$$
 Now, since $\Re (\rho(j/n_0))$ and $\Im (\rho(j/n_0))$ are nonnegative, we have, by taking the argument
 in $[0, \pi/2]$,
 $$
 0 \leq \frac{\pi}{2} - \operatorname{Arg} (\rho (j/n_0)) \leq \frac{\pi}{2} \Re (\rho(j/n_0)) \leq \frac{\mu}{\sqrt{n_0 - j}},
 $$
 and since $\rho^2 (j/n_0) =  - e^{-2 i (\pi/2 - \operatorname{Arg} (\rho (j/n_0)))}$,
 $$|1 + \rho^2 (j/n_0)| = | 1 - e^{-2 i (\pi/2 - \operatorname{Arg} (\rho (j/n_0)))}|
 \leq 2 \left| \frac{\pi}{2} -\operatorname{Arg} (\rho (j/n_0)) \right| \leq \frac{2 \mu}{\sqrt{n_0 - j}}.$$
 Hence,
 $$ \left| \sum_{\ell_1 \leq j \leq \ell_2 - 1, j < \tilde{\ell}} \Re(g_{2,j} \eta^2_j ) \right|
 = O \left( \sum_{j=0}^{\tilde{\ell}-1} \frac{\mu}{(n_0 - j)^{3/2}} \right)
  = O \left( \sum_{k = n_0 - \tilde{\ell} + 1}^{\infty}  \frac{\mu}{k^{3/2}} \right)
  = O \left( \frac{ \mu}{ \sqrt{n_0 - \tilde{\ell} + 1}} \right).$$
 Now, if $\tilde{\ell} > 0$, $n_0 - \tilde{\ell} + 1  \geq n_0 - (n_0 - \mu^2) + 1 \geq \mu^2$, and then
 $$  \sum_{j=0}^{\tilde{\ell} - 1} \Re(g_{2,j} \eta^2_j)  = O(1).$$
Adding all the estimates we have obtained on the previous sums involving real parts, we deduce
 $$ \mathbb{E}[R_{\ell_2} - R_{\ell_1}|\mathcal{F}_{\ell_1}]  =  O(1),$$
 as soon as $\ell^* \leq \ell_1 < \ell_2 \leq \ell - 1$. Since the phases $\varphi$ vary by $O(1)$ at each step, we
 can relax the assmption to $\ell^* \leq \ell_1 \leq \ell_2 \leq \ell$.
By Proposition 19 of \cite{VV}, the integer part of
 $R_j/2 \pi = (\varphi_{j, \Lambda, \mu} - \varphi_{j, \mu, \mu})/2\pi$ (which corresponds to $\alpha_{j, \lambda}/2 \pi$ with the notation of \cite{VV})  is nondecreasing in $j$, which implies
 that $R_{j_2} \geq R_{j_1} - 2\pi$ as soon as $j_2 \geq j_1$.
 Let $A$ be a strictly positive integer. For
 $\ell^* \leq \ell_1 < \ell$,
 $$\mathbb{P} [\sup_{\ell_1 < j \leq \ell} (R_j - R_{\ell_1}) \geq 2 A \pi \, | \mathcal{F}_{\ell_1}]
 \leq \mathbb{P} [ R_{\ell} - R_{\ell_1} \geq  2 (A-1) \pi \, | \mathcal{F}_{\ell_1}]
 = \mathbb{P} [ R_{\ell} - R_{\ell_1} + 2 \pi  \geq  2 A \pi \, | \mathcal{F}_{\ell_1}].$$
 The variable $R_{\ell} - R_{\ell_1} + 2 \pi$ is non-negative, so by Markov inequality,
 $$ \mathbb{P} [\sup_{\ell_1 < j \leq \ell} (R_j - R_{\ell_1}) \geq 2 A \pi \, | \mathcal{F}_{\ell_1}]
 \leq \frac{1}{2 A \pi}
 \mathbb{E} [ R_{\ell} - R_{\ell_1} + 2 \pi   \, | \mathcal{F}_{\ell_1}],$$
 which is $O(1/A)$ by the boundedness of  $ \mathbb{E}[R_{\ell_2} - R_{\ell_1}|\mathcal{F}_{\ell_1}]$.
 Hence, one can find $A = O(1)$ such that
 $$ \mathbb{P} [\sup_{\ell_1 < j \leq \ell} (R_j - R_{\ell_1}) \geq 2 A \pi \, | \mathcal{F}_{\ell_1}]
 \leq \frac{1}{2}.$$
 Now, let $(T_k)_{k \geq 1}$ be the increasing sequence of indices defined as follows: $T_0 = \ell^*$, and
 for all $k \geq 1$,
 $$T_k = \inf \{j, T_{k-1} < j \leq \ell, R_j \geq R_{T_{k-1}} + 2 A \pi\}.$$
 For $\ell^* \leq \ell_1 \leq \ell$, $k \geq 1$,
 $$ \mathbb{P} [T_{k} < \infty | T_{k-1} = \ell_1]
 = \mathbb{P} [\sup_{\ell_1 < j \leq \ell} (R_j - R_{\ell_1}) \geq 2 A \pi \, | T_{k-1} = \ell_1]
 \leq \frac{1}{2}$$
 since the event $\{T_{k-1} = \ell_1\}$ is $\mathcal{F}_{\ell_1}$-measurable ($T_{k-1}$ is a stopping time).
 Hence,
 $$  \mathbb{P} [T_{k} < \infty | T_{k-1} < \infty] \leq \frac{1}{2},$$
 which, by induction, implies that $\mathbb{P} [T_{k} < \infty] \leq 2^{-k}$.
 On the other hand, let us observe that each increment of
 $ R_j$ is the sum of two angular
 shifts, and then it is at most $4 \pi$. If $k \geq 1$ and $T_{k} < \infty$, we have
 $R_{T_{k} - 1} \leq R_{T_{k-1}} + 2A \pi$ by the minimal property of $T_k$, and then
 $$ R_{T_{k}} \leq R_{T_{k-1}} + 2(A+2) \pi,$$
 which by induction, implies
 $$ R_{T_k} - R_{\ell^*} \leq 2 k (A+2) \pi.$$
 Now, if $K$ is the first index such that $T_K = \infty$, we have $T_{K-1} \leq \ell < T_{K}$, and then
 $$R_{\ell} - R_{\ell^*} \leq R_{T_{K-1}} -  R_{\ell^*} + 2 A \pi \leq 2 (K-1) (A+2) \pi + 2 A \pi \leq 2 K (A+2) \pi.$$
 We deduce
 \begin{align*} \mathbb{E} [ (R_{\ell} - R_{\ell^*} )^2]
 \leq 4  \pi^2 (A+2)^2 \mathbb{E}[K^2]
& = 4 \pi^2 (A+2)^2  \sum_{k=1}^{\infty} (2k-1) \mathbb{P} [K \geq k].
\\ & = 4 \pi^2 (A+2)^2  \sum_{k=1}^{\infty} (2k-1) \mathbb{P} [T_{k-1} < \infty].
\\ & = 4 \pi^2 (A+2)^2 \sum_{k=1}^{\infty} (2k-1) 2^{1-k} = O(1),
\end{align*}
since $A = O(1)$ and the series in $k$ is convergent.
This proves \eqref{boundRell}, and then finishes the proof of \eqref{boundremainingterm}, and finally the proof of \eqref{boundphi2}.
\end{proof} 
\subsection{Proof of \eqref{boundtphh}}
\begin{proof} 
In this subsection, we prove the estimate: 
$$\mathbb{E} [(\tphh_{\ell, \mu} + 2 \pi (n-\ell))^2] = O(1).$$
For this purpose, let us consider the eigenvalue equations corresponding to the bottom-right 
$(n - \ell) \times (n- \ell)$ minor of the matrix $\widetilde{M}$, denoted $\widetilde{M}'$. 
These equations are the $n-\ell$ last equations corresponding to the eigenvalues of 
$\widetilde{M}$, restricted to vectors of $\mathbb{C}^n$ whose $\ell$ first components are equal to zero. 
We deduce that $\Lambda$ is an eigenvalue of $\widetilde{M}'$ if and only if 
$$\infty. \mathbf{R}_{\ell, \Lambda} .  \mathbf{R}_{\ell+1, \Lambda} \dots  \mathbf{R}_{n-1, \Lambda} = 0,$$
i.e. 
$$ 0.  \mathbf{R}^{-1}_{n-1, \Lambda} \dots  \mathbf{R}^{-1}_{\ell, \Lambda} = \infty,$$
or 
$$  \tphh_{\ell, \Lambda} = 0 \lstar \mathbf{R}^{-1}_{n-1, \Lambda} \dots  \mathbf{R}^{-1}_{\ell, \Lambda} \in \pi + 2 \pi \mathbb{Z}.$$
Now,  $\tphh_{\ell, \Lambda}$, which tends to $0$ at $-\infty$,  is continuous and strictly decreasing in $\Lambda$. Hence, the number of eigenvalues of 
$\widetilde{M}'$ which are smaller than or equal to $\mu$ is equal to the number of odd multiples of $\pi$ in the interval $[\tphh_{\ell, \Lambda}, 0)$.
We deduce that 
$$N' (-\infty, \mu) = - \frac{ \tphh_{\ell, \Lambda}}{ 2 \pi} + O(1),$$
where $N'(a,b)$ denotes the number of eigenvalues of $\widetilde{M}'$ in the interval $(a,b)$. 
Since $\widetilde{M}'$ has $n-\ell$ eigenvalues, we deduce 
$$N'(\mu, \infty) = n- \ell + \frac{ \tphh_{\ell, \Lambda}}{ 2 \pi} + O(1).$$
Hence,  \eqref{boundtphh} is proven if we show that 
$$\mathbb{E} [ ( N'(\mu, \infty)  )^2 ] = O(1).$$
 Moreover, for $n' := n- \ell$,  we have 
$$n' \leq  n - n_0 + 1 + |\mu|^{2/3} \leq n - \left(n - \frac{\mu^2}{4} - \frac{1}{2}  \right) +
1 + |\mu|^{2/3} =  \frac{\mu^2}{4} + \frac{3}{2} + |\mu|^{2/3}.$$
If we assume $\mu \leq 2 \sqrt{ n'}$, we deduce
$$ n' \leq  \frac{\mu^2}{4} + \frac{3}{2} + (4n')^{1/3},$$
and then
$$ \mu \geq  2 \sqrt{ \left(n' -  \frac{3}{2} - (4n')^{1/3} \right)_+},$$
which gives
$$\mu \geq 2 \sqrt{n'} - O((n')^{-1/6}).$$
Hence, \eqref{boundtphh} is proven if we show that for all fixed $A > 0$, 
$$\mathbb{E} [ ( N'( 2 \sqrt{n'} - A (n')^{-1/6} , \infty)  )^2 ] = O(1).$$
Now, looking carefully at the entries of $\widetilde{M}'$, we see that 
this matrix has the same law as $\widetilde{M}$ after replacing $n$ by $n' = n- \ell$.
Hence, replacing $n'$ by $n$ in the reasoning, we see that it is enough to show that 
\begin{equation}
\mathbb{E} [ ( N_n( 2 \sqrt{n} - A n^{-1/6} , \infty)  )^2 ] = O(1), \label{xxyz}
\end{equation} 
independently of $n$, where $N_n(a,b)$ is the number of eigevalues of $\widetilde{M}$ in the interval $(a,b)$, or equivalently, the number of eigenvalues of $M$ in the same interval. 
 
 We first observe that this bound is satisfied if we take the limit when $\beta$ tends to infinity. Indeed, for $\beta \rightarrow \infty$,  $M$ tends in probability to the deterministic real symmetric matrix 
$H$  such that $H_{p,p+1} = H_{p+1,p} =  \sqrt{n-p}$ for $1 \leq p \leq n-1$, and $H_{p,q} = 0$ for $|p - q| \neq 1$. 
Now, the eigenvalues of $H$ are the zeros of the Hermite polynomial of degree $n$, and classical estimates   show that the number of zeros which are larger than or equal to 
$2 \sqrt{n} - A n^{-1/6} $ is uniformly bounded when $A$ is fixed.  Moreover, it grows at most like $(1 + A)^{3/2}$ when $A > 0$ varies. 
Let us now prove that \eqref{xxyz} is a consequence of the following fact: there exists a random variable $C > 0$, whose $L^3$ norm is finite and bounded independently of $n$ (but not independently of $\beta$), such that 
\begin{equation}
 (2 \sqrt n I_n - M) \geq (1/2) ( 2 \sqrt n I_n - H) - C n^{-1/6} I_n,  \label{xxyzt} 
\end{equation}
the inequality meaning that the difference between the two sides is a positive real symmetric matrix. 
Indeed, the number of eigenvalues of $M$ which are larger than $ 2 \sqrt{n} - A n^{-1/6}$ is equal to the number of eigenvalues of the $ 2 \sqrt n I_n - M$ which are 
smaller than  $A n^{-1/6}$, and then it is at most the number of eigenvalues of $ (1/2) ( 2 \sqrt n I_n - H) - C n^{-1/6} I_n$ which are 
smaller than $A n^{-1/6}$, i.e. the number of eigenvalues of $H$ which are larger than $2 \sqrt{n}  - 2 (A + C) n^{-1/6}$. Hence, 
$$N_n( 2 \sqrt{n} - A n^{-1/6} , \infty) = O ((1  + A + C)^{3/2}), $$
which proves \eqref{xxyz} for fixed $A$, since $C$ is bounded in $L^3$.  
It is now enough to show \eqref{xxyzt}. 
This inequality can be rewritten as
$$ E \leq \sqrt{n} I_n - (H/2) +  C n^{-1/6} I_n,$$
where $E := M-H$, i.e. for all reals $(x_p)_{1 \leq p \leq n}$,
\begin{align*}
 & \sum_{1 \leq p \leq n} g_p x_p^2  + 2 \sum_{1 \leq p \leq n-1} h_p x_p x_{p+1}
 \leq \sqrt{n} \sum_{1 \leq p \leq n}  x_p^2 -  \sum_{1 \leq p \leq n-1} \sqrt{n-p} (x_p x_{p+1} )
+ C n^{-1/6}  \sum_{1 \leq p \leq n}  x_p^2
\\ & = \frac{1}{2}  \sum_{0 \leq p \leq n} \sqrt{n-p} (x_p - x_{p+1} )^2
+ \frac{1}{2} \sum_{1 \leq p \leq n} (2\sqrt{n} - \sqrt{n-p+1} - \sqrt{n-p})   x_p^2 + C n^{-1/6}  \sum_{1 \leq p \leq n}  x_p^2,
\end{align*}
where $x_0 = x_{n+1} = 0$ by convention, $g_p$ are i.i.d. Gaussian variables of variance $2/\beta$, and $h_p$ are independent variables, $h_p$ having the law of $(\chi_{\beta (n-p)}/\sqrt{\beta}) - (\sqrt{n-p})$.
Since
$$2\sqrt{n} - \sqrt{n-p+1} - \sqrt{n-p} \geq \sqrt{n} - \sqrt{n-p} = \frac{p} { \sqrt{n} +\sqrt{n-p} } \geq \frac{p}{2 \sqrt{n}},$$
  \eqref{xxyzt} is implied by the following estimates:
\begin{equation} \label{boundH1}
\sum_{1 \leq p \leq n} g_p x_p^2  \leq
 \frac{1}{4}  \sum_{0 \leq p \leq n} \sqrt{n-p} (x_p - x_{p+1} )^2
+ \frac{1}{8}  \sum_{1 \leq p \leq n} (p/\sqrt{n})  x_p^2
+  C_1 n^{-1/6}  \sum_{1 \leq p \leq n}  x_p^2
\end{equation}
and
\begin{equation} \label{boundH2}
2 \sum_{1 \leq p \leq n-1} h_p x_p x_{p+1} \leq
 \frac{1}{4}  \sum_{0 \leq p \leq n} \sqrt{n-p} (x_p - x_{p+1} )^2
+ \frac{1}{8}  \sum_{1 \leq p \leq n} (p/\sqrt{n}) x_p^2
+  C_2 n^{-1/6}  \sum_{1 \leq p \leq n}  x_p^2,
\end{equation}
for some $C_1, C_2$ bounded in $L^3$.
It is then enough to prove  \eqref{boundH1} and  \eqref{boundH2}.

\textbf{Proof of \eqref{boundH1}:} We decompose $g_p$ as $\bar{g}_p + \tilde{g}_p$, where
$$ \bar{g}_p := \frac{1}{m(n)}  \sum_{r = 1}^{m(n)} g_{p+r}, \;  \tilde{g}_p = g_p -  \bar{g}_p,$$
with $m(n) = \lfloor n^{1/3} \rfloor$, and by convention, $g_p = 0$ for $p > n$.
Let us now define for $k, \ell \geq 0$,
$$b_k = \sum_{p=1}^k g_p,$$
$$s_{\ell} := \sup_{\ell m(n) \leq k \leq (\ell + 1) m(n)}  |b_k - b_{\ell m(n)}|.$$
We have, for $\ell m(n) \leq p\leq (\ell + 1) m(n)$,
$$| \bar{g}_p| = \frac{| b_{p+m(n)} - b_p|}{m(n)} \leq  \frac{| b_{p+m(n)} -  b_{(\ell +1) m(n)} | +
|   b_{(\ell +1) m(n)}- b_{\ell m(n)}| + | b_{\ell m(n)} - b_p|}{m(n)},
$$
and then
$$| \bar{g}_p|  \leq \frac{1}{m(n)} ( 2 s_{\ell} + s_{\ell + 1}),$$
which implies
\begin{align*}
\sum_{ 1 \leq p \leq n}   \bar{g}_p x_p^2 &  \leq  \frac{1}{m(n)}
 \sum_{ 1 \leq p \leq n}   (2 s_{\lfloor p/m(n) \rfloor} + s_{1 +\lfloor p/m(n) \rfloor} )  x_p^2
\\ & \leq  \sum_{ 1 \leq p \leq n}  \left( \frac{p}{16\sqrt{n}} + C_3 n^{-1/6} \right) x_p^2,
\end{align*}
where
$$C_3  := n^{1/6}  \sup_{0 \leq \ell \leq n/m(n)} \left( \frac{  2 s_{\ell} + s_{\ell + 1}}{m(n)} - \frac{\ell m(n)}{16\sqrt{n}} \right)_+.$$
By using Doob's inequality and the fact that the variables $g_p$ are i.i.d., centered and Gaussian, we get for $q > 1$,
$$\mathbb{E} [ |s_{\ell}|^q] \ll_q \mathbb{E} [|b_{m(n)}|^q] \ll_{\beta,q} (m(n))^{q/2},$$
and then for all $ A > 1$,
\begin{align*}
 \left[  \frac{2 s_{\ell} + s_{\ell + 1}}{m(n)} - \frac{\ell m(n)}{16\sqrt{n}}
> A n^{-1/6} \right]
& \leq \left( A n^{-1/6} m(n) + \frac{\ell (m(n))^2}{16 \sqrt{n}}  \right)^{-q}  \mathbb{E} [
|2 s_{\ell} + s_{\ell + 1} |^{q} ]
\\ & \ll_{\beta,q} ( n^{1/6}(A + \ell))^{-q} n^{q/6}  = (A+\ell)^{-q}
\end{align*}
and
$$\mathbb{P} [C_3 \geq A]
\ll_{\beta,q} \sum_{\ell = 0}^{\infty}
(A+ \ell)^{-q} \leq
\int_{A-1}^{\infty} x^{-q}
\ll_{q} (A-1)^{1-q}.$$
Taking $q$ large enough, we deduce that $C_3$ is bounded in $L^3$. In order to prove \eqref{boundH1}, it  is then sufficient to check
$$
\sum_{1 \leq p \leq n} \tilde{g}_p x_p^2  \leq
 \frac{1}{4}  \sum_{0 \leq p \leq n} \sqrt{n-p} (x_p - x_{p+1} )^2
+ \frac{1}{16}  \sum_{1 \leq p \leq n} (p/\sqrt{n})  x_p^2
+  C_4 n^{-1/6}  \sum_{1 \leq p \leq n}  x_p^2 $$
for some $C_4$ bounded in $L^3$. Summation by parts gives
$$\sum_{1 \leq p \leq n} \tilde{g}_p x_p^2
=\sum_{1 \leq p \leq n} (\tilde{b}_p -
\tilde{b}_{p-1})  x_p^2
= \sum_{1 \leq p \leq n} \tilde{b}_p (x_p^2 - x_{p+1}^2),$$
where
$$\tilde{b}_p := \sum_{k = 1}^p \tilde{g}_k.$$
Now, for $p \leq n/2$,
\begin{align*}\tilde{b}_p (x_p^2 - x_{p+1}^2)
&=  (1/2)(n-p)^{1/4} (x_p - x_{p+1}) \tilde{b}_p (2(n-p)^{-1/4}) (x_p + x_{p+1})
\\ & \leq \frac{1}{8} (n-p)^{1/2} (x_p - x_{p+1})^2
+ 2 \tilde{b}^2_p (n-p)^{-1/2} (x_p + x_{p+1})^2
\\ & \leq \frac{1}{8} \sqrt{n-p} (x_p - x_{p+1})^2
+ (\sqrt{32/n} )\tilde{b}^2_p  (x_p^2 + x_{p+1}^2),
\end{align*}
whereas, for $p > n/2$,
$$ \tilde{b}_p (x_p^2 - x_{p+1}^2)
\leq |\tilde{b}_p| (x_p^2 + x_{p+1}^2).$$
Hence,
$$\sum_{1 \leq p \leq n} \tilde{g}_p
x_p^2 \leq \frac{1}{8} \sum_{1 \leq p \leq n}
\sqrt{n-p} (x_p - x_{p+1})^2
+ \sum_{1 \leq p \leq n}
\left( (\sqrt{32/n} )(\tilde{b}_p^2
+ \tilde{b}_{p-1}^2) + (|\tilde{b}_p|
+ |\tilde{b}_{p-1}|) \mathds{1}_{p>n/2}
\right) x_p^2.$$
We deduce that \eqref{boundH1} is proven if we check that for $1 \leq p \leq n$,
$$(\sqrt{32/n} )(\tilde{b}_p^2
+ \tilde{b}_{p-1}^2) + (|\tilde{b}_p|
+ |\tilde{b}_{p-1}|) \mathds{1}_{p>n/2}
\leq p/(16 \sqrt{n}) + C_4 n^{-1/6}$$
for $C_4$ bounded in $L^3$.
Now,
$$\tilde{b}_k = \sum_{p=1}^k
\left( g_p -  \frac{1}{m(n)}
\sum_{r = 1}^{m(n)} g_{p+r} \right)
= b_k - \frac{1}{m(n)}
\sum_{r = 1}^{m(n)} (b_{k+r} - b_r)
= \frac{1}{m(n)}\sum_{r = 1}^{m(n)}
(b_{k+r} - b_k - b_r).
$$
We then get, for $\ell m(n) \leq k \leq
(\ell +1) m(n)$,
$$ |\tilde{b}_k| \leq s_0 + 2s_{\ell} +
s_{\ell + 1}\leq 2 (s_0 + s_{\ell} +
s_{\ell + 1} )$$
and then for $\ell m(n) + 1 \leq k \leq
(\ell +1) m(n)$,
$$ |\tilde{b}_k| + |\tilde{b}_{k-1}|
\leq 4(s_0 + s_{\ell} +
s_{\ell + 1}), \; \tilde{b}_k^2 + \tilde{b}_{k-1}^2
\leq 48 (s_0^2 + s_{\ell}^2 + s_{\ell + 1}^2).$$
It is then enough to show, for $0 \leq \ell
\leq n/m(n)$,
$$48 \sqrt{32/n} (s_0^2 + s_{\ell}^2 +
s_{\ell+1}^2) +
4 (s_0 + s_{\ell} + s_{\ell+1} ) \mathds{1}_{\ell > (n/2m(n)) - 1}
\leq \frac{\ell m(n)}{16\sqrt{n}} + C_4 n^{-1/6},$$
which is guaranteed if we check
$$s_{\ell} \leq \frac{1}{1000}( \sqrt{\ell} + C_5) n^{1/6},$$
and
$$s_{\ell} \leq \frac{\sqrt{n}}{1000} +
C_6 n^{-1/6},$$
for $C_5$ and $C_6$ bounded in $L^6$.
Now, for $A > 1$, $q > 1$, we deduce from the previous estimate of the moments of $s_{\ell}$,
$$\mathbb{P} \left[s_{\ell} > \frac{1}{1000}( \sqrt{\ell} + A) n^{1/6}
\right] \ll_{\beta,q}
(\sqrt{\ell} + A)^{-q},
$$
and then for $q > 2$,
\begin{align*}
\mathbb{P} \left[ \sup_{0 \leq \ell \leq n/m(n)}  s_{\ell} \right. & \left.  > \frac{1}{1000}( \sqrt{\ell} + A) n^{1/6} \right]
 \ll_{\beta,q} \sum_{\ell \geq 0}
(\sqrt{\ell} + A)^{-q}
\leq A^{-q} + \int_{0}^{\infty} (\sqrt{x}
+ A)^{-q} dx
\\ & \leq A^{-q} + 2\int_{0}^{\infty}
y (y+A)^{-q} dy
\ll A^{-q} + \int_{A}^{\infty} t^{1-q}
dt \ll_q A^{2-q},
\end{align*}
and
\begin{align*}
\mathbb{P} \left[ \sup_{0 \leq \ell \leq n/m(n)}  s_{\ell} \right. & \left.   > \frac{\sqrt{n}}{1000} + A n^{-1/6} \right]
\ll_{\beta,q} (1 + (n/m(n)))
\left(\frac{\sqrt{n}}{1000} + A n^{-1/6}
\right)^{-q} n^{q/6}
\\ & \ll_q n^{2/3} \left( n^{1/3} +
A n^{-1/3} \right)^{-q}
\leq n^{2/3}( n^{1/3})^{-q/2}  \left( n^{1/3} +
A n^{-1/3} \right)^{-q/2}
\\ & \leq n^{2/3 - (q/6)}\left( \sqrt{n^{1/3}(A n^{-1/3}) }  \right)^{-q/2}
\leq A^{-q/4}
 \end{align*}
if $q \geq 4$. This completes the proof of
\eqref{boundH1}.

\textbf{Proof of \eqref{boundH2}:} We use, with obvious notation,
the similar decomposition $h_p =
\bar{h}_p + \tilde{h}_p$, and we denote
$$b'_k :=
\sum_{p=1}^k h_p,
\; s'_{\ell} :=
\sup_{\ell m(n) \leq k \leq (\ell +1) m(n)} |b'_k - b'_{\ell m(n)}|$$
with the convention $h_p = 0$ for $p > n-1$. We get as above
\begin{align*}\sum_{1 \leq p \leq n-1}
2 \bar{h}_p x_p x_{p+1}
\leq \sum_{1 \leq p \leq n-1}
|\bar{h}_p|( x_p^2 +  x_{p+1}^2)
& \leq
\sum_{1 \leq p \leq n-1}
\left( \frac{p}{32 \sqrt{n}} + C_7
n^{-1/6}  \right) (x_p^2 + x_{p+1}^2)
\\ & \leq \sum_{1 \leq p \leq n}
\left( \frac{p}{16 \sqrt{n}} + 2 C_7
n^{-1/6}  \right) x_p^2,
\end{align*}
where
$$C_7 := n^{1/6}
\underset{0 \leq \ell \leq n/m(n)}{\sup}
\left( \frac{2 s'_{\ell} + s'_{\ell +1}}{m(n)} - \frac{\ell m(n)}{32 \sqrt{n}}
\right)_+.$$
Since $b'_{\ell}$ is given by a sum of integrable, independent random variables,
$(b'_{p} - \mathbb{E}[b'_{p}])_{p \geq 0}$ is a martingale. On the other hand,
\begin{align*}s'_{\ell}  \leq
\sup_{\ell m(n) \leq k \leq (\ell + 1)
m(n)} & |b'_k - \mathbb{E}[b'_k]
- b'_{\ell m(n)} + \mathbb{E}[b'_{\ell m(n)}]| +
\sup_{\ell m(n) \leq k \leq (\ell + 1)
m(n)} |\mathbb{E}[b'_k]  - \mathbb{E}[
 b'_{\ell m(n)}]|
\\ & \leq \sup_{\ell m(n) \leq k \leq (\ell + 1)
m(n)} |b'_k - \mathbb{E}[b'_k]
- b'_{\ell m(n)} + \mathbb{E}[b'_{\ell m(n)}]| + \delta_{\ell, n},
\end{align*}
where
 $$\delta_{\ell, n}
 = \sum_{\ell m(n) < p \leq (\ell + 1) m(n)} |\mathbb{E}[h_p]|$$
 By using Doob's inequality, we deduce, for $q > 1$,
 \begin{align*}
 \mathbb{E}[|s'_{\ell}|^q]
 \ll_{q}
 \mathbb{E} [|b'_{(\ell + 1)m(n)} & -
 \mathbb{E} [b'_{(\ell + 1)m(n)}]
 - b'_{\ell m(n)} + \mathbb{E}
[ b'_{\ell m(n)} ]|^q] + \delta_{\ell,n}^q
\end{align*}
 Now, the expectation of a $\chi$ variable satisfies, by log-convexity of the Gamma function:
 $$ \sqrt{u} - \frac{1}{\sqrt{u}} \leq \sqrt{(u-1)_+} \leq \mathbb{E}[\chi_u] = \sqrt{2}
 \frac{\Gamma ((u+1)/2)}{\Gamma(u/2)}
 \leq \sqrt{u}.$$
and then $|\mathbb{E}[h_p]| \ll_{\beta}
(n-p)^{-1/2}$ for $1 \leq p \leq n-1$.
We deduce that $\delta_{\ell,n}$ is dominated by the sum of $m(n)$ consecutive inverse square roots of integers, and then
$$\delta_{\ell,n}^q \ll_{q,\beta}
\left(\sum_{j=1}^{m(n)} j^{-1/2}  \right)^q
\ll_q  (m(n))^{q/2}.$$
 On the other hand, by Rosenthal's inequality, for $q > 2$,
 \begin{align*}
 & \mathbb{E} [|b'_{(\ell + 1)m(n)}  -
 \mathbb{E} [b'_{(\ell + 1)m(n)}]
 - b'_{\ell m(n)} + \mathbb{E}
[ b'_{\ell m(n)} ]|^q]
\\ &  \ll_q
\sum_{\ell m(n) < p \leq (\ell+1) m(n)}
\mathbb{E}[ |h_p - \mathbb{E}[h_p]|^q]
+\left( \sum_{\ell m(n) < p \leq (\ell+1) m(n)}
\mathbb{E}[ \operatorname{Var}(h_p)] \right)^{q/2}.
\end{align*}
Now, for all $u > 0$,
\begin{align*}\mathbb{E}[|\chi_u - \mathbb{E}[\chi_u]|^q]
& \ll_q\mathbb{E}[|\chi_u - \sqrt{u}|^q]
+ |\sqrt{u} - \mathbb{E}[\chi_u]|^q
\\ & \ll_q \mathbb{E}[|\chi_u^2 - u|^q
|\chi_u + \sqrt{u}|^{-q}] +
|\sqrt{u} - \sqrt{(u-1)_+}|^q
\\ & \ll_q u^{-q/2} \mathbb{E}[|\chi_u^2 - u|^q] + 1
\end{align*}
Now, $\chi_u^2 -  u$ can be written as sum of $\lceil u \rceil$ independent random variables of the form $\chi_v^2 - v$ where $0 \leq v \leq 1$. These variables are centered, with uniformly bounded moments of order $r$ for fixed $r \geq 2$, which implies, again by using Rosenthal's inequality:
$$\mathbb{E} [|\chi_u^2 - u|^q]
\ll_q  \lceil u \rceil + (\lceil u \rceil)^{q/2} \ll (1+u)^{q/2},$$
then
$$\mathbb{E}[|\chi_u - \mathbb{E}[\chi_u]|^q]  \ll_q ( 1+ u^{-1})^{q/2}, \; \mathbb{E} [|h_p - \mathbb{E}[h_p]|^q]
\ll_{q, \beta} 1,$$
$$\mathbb{E} [|b'_{(\ell + 1)m(n)}  -
 \mathbb{E} [b'_{(\ell + 1)m(n)}]
 - b'_{\ell m(n)} + \mathbb{E}
[ b'_{\ell m(n)} ]|^q] \ll_{q, \beta}
(m(n))^{q/2},$$
and finally
$$\mathbb{E}[|s'_{\ell}|^q] \ll_q (m(n))^{q/2}.$$
With this estimate, we deduce that $C_7$ is bounded in $L^3$, similarly as in the proof of \eqref{boundH1}.
It is now sufficient to check
$$
\sum_{1 \leq p \leq n-1} 2 \tilde{h}_p x_p
x_{p+1} \leq
 \frac{1}{4}  \sum_{0 \leq p \leq n} \sqrt{n-p} (x_p - x_{p+1} )^2
+ \frac{1}{16}  \sum_{1 \leq p \leq n} (p/\sqrt{n})  x_p^2
+  C_8 n^{-1/6}  \sum_{1 \leq p \leq n}  x_p^2 $$
for some $C_8$ bounded in $L^3$. Summation by parts here gives
$$\sum_{1 \leq p \leq n-1} \tilde{h}_p x_p x_{p+1}
=\sum_{1 \leq p \leq n-1} (\tilde{b}'_p -
\tilde{b}'_{p-1})  x_p x_{p+1}
= \sum_{1 \leq p \leq n-1} \tilde{b}'_p x_{p+1} (x_p - x_{p+2}),$$
where
$$\tilde{b}'_p := \sum_{k = 1}^p \tilde{h}_k.$$
Now, for $p \leq n/2$,
\begin{align*}\tilde{b}'_p x_{p+1} (x_p - x_{p+2})
&=  (1/10)(n-p)^{1/4} (x_p - x_{p+2}) \tilde{b}'_p (10(n-p)^{-1/4}) x_{p+1}
\\ & \leq \frac{1}{200} (n-p)^{1/2} (x_p - x_{p+2})^2
+ 50 (\tilde{b}'_p)^2 (n-p)^{-1/2} x_{p+1}^2
\\ & \leq \frac{1}{100} \sqrt{n-p} [(x_p - x_{p+1})^2 + (x_{p+1}- x_{p+2})^2]
+ (50\sqrt{2/n} )(\tilde{b}'_p)^2  x_{p+1}^2,
\end{align*}
whereas, for $p > n/2$,
$$ \tilde{b}'_p x_{p+1} (x_p - x_{p+2})
\leq |\tilde{b}'_p| (x_p^2 + x_{p+1}^2 + x_{p+2}^2).$$
Since for $p \leq n/2$,
$$  \frac{1}{100} (\sqrt{n-p} + \sqrt{n-p+1})
\leq  \frac{1}{100} (\sqrt{n-p} + \sqrt{(n-p)(1+(2/n))})
\leq \frac{1+\sqrt{3}}{100} \sqrt{n-p}$$
we deduce,
\begin{align*}
\sum_{1 \leq p \leq n-1} \tilde{h}_p
x_p  x_{p+1} & \leq \frac{1+\sqrt{3}}{100} \sum_{1 \leq p \leq n}
\sqrt{n-p} (x_p - x_{p+1})^2
\\ & + \sum_{1 \leq p \leq n}
\left( (50 \sqrt{2/n} )(\tilde{b}'_{p-1})^2  + (|\tilde{b}'_p|
+ |\tilde{b}'_{p-1}|+ |\tilde{b}'_{p-2}|) \mathds{1}_{p>n/2}
\right) x_p^2
\end{align*}
(with the convention $\tilde{b}'_{-1}= 0$ for $n = p = 1$).
We deduce that \eqref{boundH2} is proven if we check that for $1 \leq p \leq n$,
$$(100 \sqrt{2/n} )(\tilde{b}'_{p-1})^2
 + 2 (|\tilde{b}'_p|
+ |\tilde{b}'_{p-1}| +|\tilde{b}'_{p-2}| ) \mathds{1}_{p>n/2}
\leq p/(16 \sqrt{n}) + C_8 n^{-1/6}$$
for $C_8$ bounded in $L^3$.

As in the proof of \eqref{boundH1}, we get, for $\ell m(n) \leq k \leq
(\ell +1) m(n)$,
$$ |\tilde{b}'_k| \leq s'_0 + 2s'_{\ell} +
s'_{\ell + 1}\leq 2 (s'_0 + s'_{\ell} +
s'_{\ell + 1} )$$
and then for $\ell m(n) + 1 \leq k \leq
(\ell +1) m(n)$,
$$ |\tilde{b}'_k|, |\tilde{b}'_{k-1}| \leq 2 (s'_0 + s'_{\ell} +
s'_{\ell + 1} ),$$
$$|\tilde{b}'_{k-2}|  \leq  \sup( 2 (s'_0 + s'_{\ell} + s'_{\ell + 1}),  2 (s'_0 + s'_{\ell-1} + s'_{\ell}))
\leq 2 ( s'_0 +s'_{\ell-1}+  s'_{\ell} + s'_{\ell + 1})$$
with the convention $s'_{-1} = 0$, which implies
$$ |\tilde{b}'_k| + |\tilde{b}'_{k-1}| +  |\tilde{b}'_{k-2}|
\leq 6(s'_0 + s'_{\ell-1} + s'_{\ell} +
s'_{\ell + 1}), \; (\tilde{b}'_{k-1})^2
\leq 12  ((s'_0)^2 + (s'_{\ell})^2 + (s'_{\ell + 1})^2).$$
It is then enough to show, for $0 \leq \ell
\leq n/m(n)$,
$$1200 \sqrt{2/n} ((s'_0)^2 + (s'_{\ell})^2 + (s'_{\ell + 1})^2) +
12 (s'_0 + s'_{\ell-1} +  s'_{\ell} + s'_{\ell+1} ) \mathds{1}_{\ell > (n/2m(n)) - 1}
\leq \frac{\ell m(n)}{16\sqrt{n}} + C_8 n^{-1/6},$$
which is guaranteed if we check
$$s'_{\ell} \leq \frac{1}{10000}( \sqrt{\ell} + C_9) n^{1/6},$$
and
$$s'_{\ell} \leq \frac{\sqrt{n}}{10000} +
C_{10} n^{-1/6},$$
for $C_9$ and $C_{10}$ bounded in $L^6$.
This last estimate is proven exactly in the same way as in the proof of \eqref{boundH1}, by using the estimate
$$\mathbb{E}[(s'_{\ell})^q] \ll_{\beta,q} (m(n))^{q/2}.$$
\end{proof}
\noindent {\bf Acknowledgments.} B.V. was supported by the Canada
Research Chair program, the NSERC Discovery Accelerator grant, the MTA 
Momentum Random Spectra research group, and the ERC consolidator grant 
648017 (Abert).
\bibliographystyle{halpha}
\bibliography{biblinv}

\begin{thebibliography}{BLS18}

\bibitem[BLS18]{BLS18}
F.~Bekerman, T.~Lebl\'e, and S.~Serfaty.
\newblock {CLT} for fluctuations of $\beta $-ensembles with general potential.
\newblock {\em Electron. J. Probab.}, 23, 2018.

\bibitem[Bou09]{Bou09}
C.~Boutillier.
\newblock The bead model and limit behaviors of dimer models.
\newblock {\em Ann. Probab.}, 37(1):107--142, 2009.

\bibitem[CL95]{CL95}
O.~Costin and J.~L. Lebowitz.
\newblock Gaussian fluctuation in random matrices.
\newblock {\em Phys. Rev. Lett.}, 75(1):69--72, 1995.

\bibitem[DE01]{DE01}
P.~Diaconis and S.~N. Evans.
\newblock Linear functionals of eigenvalues of random matrices.
\newblock {\em Trans. Amer. Math. Soc.}, 353(7):2615--2633, 2001.

\bibitem[DE02]{DE02}
I.~Dumitriu and A.~Edelman.
\newblock Matrix models for beta-ensembles.
\newblock {\em J. Math. Phys}, pages 5830--5847, 2002.

\bibitem[DS94]{DS94}
P.~Diaconis and M.~Shahshahani.
\newblock On the eigenvalues of random matrices.
\newblock {\em J. Appl. Probab.}, 31A:49--62, 1994.
\newblock Studies in applied probability.

\bibitem[Gus05]{Gus05}
J.~Gustavsson.
\newblock Gaussian fluctuations of eigenvalues in the {GUE}.
\newblock {\em Ann. Inst. H. Poincar\'{e} Probab. Statist.}, 41(2):151--178,
  2005.

\bibitem[HL18]{Huang2018}
J.~Huang and B.~Landon.
\newblock Rigidity and a mesoscopic central limit theorem for dyson brownian
  motion for general $\beta$ and potentials.
\newblock {\em Probability Theory and Related Fields}, 2018.

\bibitem[Hua19]{Huang19}
J.~Huang.
\newblock Eigenvalues for the {M}inors of {W}igner {M}atrices.
\newblock http://arxiv.org/pdf/1907.10214, 2019.

\bibitem[JM15]{JM15}
T.~Jiang and S.~Matsumoto.
\newblock Moments of traces of circular beta-ensembles.
\newblock {\em Ann. Probab.}, 43(6):3279--3336, 2015.

\bibitem[Joh98]{Joh98}
K.~Johansson.
\newblock On fluctuations of eigenvalues of random {H}ermitian matrices.
\newblock {\em Duke Math. J.}, 91(1):151--204, 1998.

\bibitem[Kil08]{Kil08}
R.~Killip.
\newblock Gaussian fluctuations for {$\beta$} ensembles.
\newblock {\em Int. Math. Res. Not. IMRN}, (8):Art. ID rnn007, 19, 2008.

\bibitem[KN04]{bib:KN04}
R.~Killip and I.~Nenciu.
\newblock Matrix models for circular ensembles.
\newblock {\em Int. Math. Res. Not.}, (50):2665--2701, 2004.

\bibitem[KS09]{bib:KSt09}
R.~Killip and M.~Stoiciu.
\newblock Eigenvalue statistics for {CMV} matrices: from {P}oisson to clock via
  random matrix ensembles.
\newblock {\em Duke Math. J.}, 146(3):361--399, 2009.

\bibitem[Nak14]{Nakano}
F.~Nakano.
\newblock Level statistics for one-dimensional {S}chr\"odinger operators and
  {G}aussian beta ensemble.
\newblock {\em Journal of Stat. Phys.}, 156(1):66--93, 2014.

\bibitem[NV19]{NV19}
J.~Najnudel and B.~Vir\'ag.
\newblock The bead process for beta ensembles.
\newblock http://arxiv.org/pdf/1904.00848, 2019.

\bibitem[Shc13]{Shc13}
M.~Shcherbina.
\newblock Fluctuations of linear eigenvalue statistics of {$\beta$} matrix
  models in the multi-cut regime.
\newblock {\em J. Stat. Phys.}, 151(6):1004--1034, 2013.

\bibitem[Sos00]{Sos00}
A.~B. Soshnikov.
\newblock Gaussian fluctuation for the number of particles in {A}iry, {B}essel,
  sine, and other determinantal random point fields.
\newblock {\em J. Statist. Phys.}, 100(3-4):491--522, 2000.

\bibitem[Sos02]{Sos02}
A.~B. Soshnikov.
\newblock Gaussian limit for determinantal random point fields.
\newblock {\em Ann. Probab.}, 30(1):171--187, 2002.

\bibitem[Tro84]{Tro84}
H.~Trotter.
\newblock Eigenvalue distributions of larger {H}ermitian matrices; {W}igner's
  semi-circle law and a theorem of {K}ac, {M}urdock and {S}zeg\"o.
\newblock {\em Adv. in Math.}, 54(1):67--82, 1984.

\bibitem[VV09]{VV}
B.~Valk{\'o} and B.~Vir{\'a}g.
\newblock Continuum limits of random matrices and the {B}rownian carousel.
\newblock {\em Invent. Math.}, 177(3):463--508, 2009.

\end{thebibliography}

\end{document}